\documentclass{amsart}
\usepackage{latexsym,amsmath,amsthm,amssymb}
\usepackage[dvips]{graphicx}

\newtheorem{theorem}{Theorem}[section]
\newtheorem{lemma}[theorem]{Lemma}
\newtheorem{corollary}[theorem]{Corollary}%[section]
\newtheorem{proposition}[theorem]{Proposition}%[section]
\newtheorem{conjecture}[theorem]{Conjecture}%[section]

\theoremstyle{remark}
\newtheorem{example}[theorem]{Example}
\newtheorem{remark}[theorem]{Remark}%[section]

\theoremstyle{definition}
\newtheorem{definition}[theorem]{Definition}

\newcommand{\dR}{\ensuremath{\mathbb{R}}} %reels
\newcommand{\R}{\dR}

\begin{document}

\title{On isoperimetric sets of radially symmetric measures}

%\centerline{\large \footnote{Higher School of Economics,  Moscow State University of Printing Arts, and St. Tikhon Orthodox University}Alexander V. Kolesnikov, \footnote{Moscow State University}{Roman I. Zhdanov}}

\author{Alexander V. Kolesnikov}
\address[1]{Moscow State University of Printing Arts, Higher School of Economics,  and St. Tikhon Orthodox University}
%\email{ledoux@math.univ-toulouse.fr}

\author{Roman I. Zhdanov}
\address{Moscow State University}


%    General info
%\subjclass{}
%\date{}

\begin{abstract}
We study the isoperimetric problem for  the radially symmetric
measures. Applying the spherical symmetrization procedure and
variational arguments we reduce this problem to a one-dimensional
ODE of the second order. Solving numerically this ODE we get an
empirical description of isoperimetric regions of the planar radially
symmetric exponential power laws. We also prove some isoperimetric
inequalities for the log-convex measures. It is shown, in particular,
that the symmetric balls of large size are isoperimetric sets for
strictly log-convex and radially symmetric measures. In addition, we establish
 some comparison results for general log-convex measures.
\end{abstract}

\maketitle

\makeatletter\renewcommand{\@makefnmark}{}\makeatother
\footnotetext{
This work was supported by the RFBR projects 07-01-00536 and 08-01-90431-Ukr,
RF President Grant MD-764.2008.1.
}

\section{Introduction}

Let $\mu$ be a Borel measure on $\R^d$ (or  a Riemannian manifold)
and $A$ be a Borel set. We consider its surface measure $\mu^{+}(\partial A)$
$$
\mu^{+}(\partial A) = \underline{\lim}_{\varepsilon \to 0} \frac{\mu(A^\varepsilon)-\mu(A)}{\varepsilon},
$$
where $A^\varepsilon$ is the $\varepsilon$-neighborhood  of $A$.
Recall that a set $A$ is called isoperimetric if it has the minimal surface measure
among of all the sets with the same measure $\mu(A)$.

The isoperimetric function  $\mathcal{I}_{\mu}(t)$ of
$\mu$ is defined by
$$
\mathcal{I}_{\mu}(t) = \inf_{A} \{ \mu^{+}(\partial A): \mu(A)=t
\}.
$$

In what follows we denote by $\mathcal{H}^{k}$ the $k$-dimensional
Hausdorff measure. For the Lebesgue measure we also use the common
notations $\lambda$ and $dx$. If $\mu$ has a continuous density
$\rho$, then the surface measure $\mu^{+}$ has
 the following representation:
$\mu^{+} = \rho \cdot \mathcal{H}^{d-1}$. We  denote by
$\kappa_{d}$ the constant appearing in the Euclidean isoperimetric
inequality $\lambda$: $\lambda^{1-\frac{1}{d}}(A) \le
\kappa_{d} \mathcal{H}^{d-1}(\partial A)$.

In this paper we study the isoperimetric sets of the
radially symmetric measures, i.e. measures with densities of the type
$$
\mu=\rho (r) \ dx = e^{-v(r)} \ dx.
$$

Only a small number of spaces with an exact solution
 to the isoperimetric problem is known so far. The most well known examples are
 \begin{itemize}
\item[1)] Euclidean space with the Lebesgue measure (solutions are the balls)
\item[2)] Spheres $S^{d-1}$ and hyperbolic spaces $H^{d-1}$ (solutions are the metric balls)
\item[3)]  Gaussian measure
$
\gamma = \frac{1}{(\sqrt{2\pi})^d}  e^{-\frac{|x|^2}{2}} dx
$
(solutions are the half-spaces).
\end{itemize}
Some other examples can be found in \cite{Ros}. See also
recent developments in \cite{RCBM}, \cite{MaurMorg}, \cite{CJQW},  \cite{EMMP}, \cite{DDNT}, \cite{CMV}.

Whereas $S^{d-1}$ and $H^{d-1}$ are the model spaces in geometry, the Gaussian measures
are the most important model measures in probability theory.
The solutions to the isoperimetric problem for the Gaussian measures have been obtained by
Sudakov and Tsirel'son  \cite{ST} (see also Borell \cite{Bor-iso1}).
The proof given in \cite{ST} used the solution to the isoperimetric problem on the sphere. Ehrhard \cite{Erh} found later another proof based
on the Steiner symmetrization for Gaussian measures (see \cite{Ros} for generalizations to product measures).

Some exact solutions to the isoperimetric problem are known in the one-dimensional case.
For instance, the half-lines are the isoperimetric sets
  for the probability log-concave one-dimensional distributions  (see \cite{BobkovHoudre}). This result was generalized in
\cite{RCBM}.

Another  interesting result
has been obtained by Borell in  \cite{Bor-iso2}. He has shown that
 the balls about the origin
$$B_R = \{|x| \le R\}$$ are solutions to the isoperimetric problem for
the non-probability measure $$\mu = e^{r^2} \ dx.$$ Several
extensions of this result can be found in \cite{RCBM},
\cite{MaurMorg}. It was conjectured in \cite{RCBM} that  the balls
about the origin are solutions to the isoperimetric problem for
$\mu=\rho(r)\ dx$ provided $\log \rho(r)$ is smooth and convex.
We deal below with  a slightly changed version of this conjecture. Namely, we are interested
in radially symmetric measures with increasing convex $r \to \log \rho(r)$ (e.g. $\mu = e^{r^{\alpha}}$, $\alpha \ge 1$).

To our knowledge,  no any other non-trivial exact
solutions to the isoperimetric problem coming from the probability
theory are known. For instance,  it was proved in \cite{RCBM} that there
exist isoperimetric regions for log-concave radially symmetric
distributions which are neither balls nor halfplanes, but no
precise example was given.

  The paper is organized as follows.
  In Section 3 we prove a symmetrization result for rotationally invariant measures.
  This result is not new. During the preparation of the
manuscript we learned from Frank Morgan about a recent
symmetrization result for warped products of manifolds in \cite{MHH} (Proposition 3, Proposition 5). See also
remarks in Section 3.2 of \cite{Ros} and Section 9.4 of \cite{Gromov}. In this
paper we provide alternative arguments which are close to the classical
proof that the Steiner symmetrization does not increase the
surface measure.

 By variational arguments we show that every stationary set for the measure $\mu = e^{-v(r)} \ dx$ on the plane which has $Ox$ as
the revolution axis and a real analytic boundary is either a ball or has the form
\begin{equation}
A = \{(r,\theta) : -f(r) < \theta < f(r) \}, \ f(r) \in [0,\pi],
\end{equation}
where $f(r)$ is a solution to
\begin{equation}
\label{1dimeq}
 \Bigl[ \frac{r^2 f' }{\sqrt{1 + r^2 (f')^2}}\Bigr]' - v'(r)  \Bigl[ \frac{r^2 f'}{\sqrt{1 + r^2 (f')^2}}\Bigr] = c \cdot r
\end{equation}
for some constant $c$.

We analyze this equation for several precise examples. It turns out that apart from special cases
(Lebesgue measure) only small part of the solutions  to this ODE can describe
 an isoperimetric set.
It looks in general impossible to determine analytically the constant $c$ and the initial conditions for
(\ref{1dimeq})
such that the corresponding solution
describes an isoperimetric region.  Nevertheless, performing numerical computations it is possible
to find {\it empirically} the desired parameters, since most
 of the solutions to
(\ref{1dimeq}) are either non-smooth or non-closed curves.

We are especially interested in is
the exponential power law $$\rho(r) = C_{\alpha} e^{-r^{\alpha}}$$
on the plane.
We justify by numerical computations that for the super-Gaussian laws $\alpha >2$ the isoperimetric regions are
non-compact
and can be obtained by a separation of the plane in two pieces by an axially symmetric convex curve.
Unlike this, the isoperimetric regions for the exponential law $\rho =C_1 e^{-r}$
are compact convex axially symmetric sets (which are not the  balls) and their complements.
For some values $\alpha \in (1,2)$ there exist isoperimetric regions of both types.

In the last section we analyze a non-probabilistic case:
$\mu=e^{V} \ dx$, where $V$ is a convex potential. The interest in
this type of measures is partially motivated by
problems coming from the differential geometry. The measures of this type are natural "flat" analogs of the
negatively curved spaces. In fact, both types of spaces enjoy very
similar isoperimetric inequalities. Note that the famous
Cartan-Hadamard conjecture on a comparison isoperimetric
inequality for the manifolds with negative sectional curvatures is still
an open problem.

We prove some results related to the cited conjecture from
\cite{RCBM}. We show, in particular, that the large balls are the
isoperimetric sets for $\mu = e^{V} \ dx$ under assumption that
$V=r^{\alpha},  \alpha>1$ (more generally, $V$ is convex, radially symmetric and superlinear). Applying mass transportational arguments
we prove that every log-convex radially symmetric measure
$\mu=e^{V} \ dx$ satisfies
$$
\mu^{+}(\partial A) \ge \frac{1}{\sqrt{1+\pi^{2}}}
\mu^{+}(\partial B)
$$
for every Borel set $A$ and a  ball $B$ about the origin
satisfying $\mu(A)=\mu(B)$.

We also prove some comparison theorems for  log-convex measures of general type.
We show, in particular, that $\mu = e^{V} dx$ with a convex non-negative $V$
enjoys the Euclidean isoperimetric inequality.
Finally, we prove some results for the products of the one-dimensional (non-probability!) log-convex measures.
The case of probability product measures has been studied in
\cite{BobkovHoudre}, \cite{Barthe01}, \cite{BartheMaurey}.
We  establish a log-convex (one-dimensional) analog of a
Caffarelli's contraction theorem for the optimal transportation of the uniformly
log-concave measures. More precise, we show that every
one-dimensional log-convex measure $\mu = e^V dx$ satisfying
$$
V'' e^{-2V} \ge 1,
$$
$V$ is even and $V(0)=0$
is a $1$-Lipschitz image of the model measure $\nu =
\frac{dx}{\cos x}.$ In particular, this implies the following
comparison result:
$$
\mathcal{I}_{\mu}(t) \ge \mathcal{I}_{\nu}(t) = e^{t/2} + e^{-t/2}.
$$
 Finally, we
estimate the isoperimetric profile for a large class of the  log-convex
product measures.

We thank Frank Morgan for reading the preliminary version of the manuscript and important remarks.
A.K. thanks Andrea Colesanti for his hospitality during the author's very nice visit
of the University of Florence where
this work was partially done.

\section{Existence, regularity and geometric properties of isoperimetric sets}

It is known that under broad assumptions the isoperimetric regions do exist for measures with a
finite  total volume. Some results on existence for measures with an infinite total volume can
be found in \cite{RCBM}.

We will widely use the fact that the isoperimetric low-dimensional
surfaces are regular. A classical result on regularity of the
isoperimetric sets was obtained by Almgren \cite{Almgren}. We use
the following refinement obtained by F.~Morgan \cite{Morgan}. The
original formulation is given in terms of a Riemannian metric, but
the result still holds for the trivial Riemannian metric and a
potential with the same regularity.

Let $A$ be an open set with smooth boundary $\partial A$ and $\{ \phi_t \}, \phi_0 = \mbox{Id}$ be any smooth family of diffeomorphisms
satisfying $\mu(A_t) = \mu(A)$, where $A_t = \phi_t(A)$. We call $A$ stationary if
$$\frac{d}{dt} \mu^{+}(\partial A_t)|_{t=0} = 0.$$

We call $A$ stable if
$$\frac{d^2}{dt^2} \mu^{+}(\partial A_t)|_{t=0} \ge 0.$$

Clearly, isoperimetric sets must be stationary and stable.

\begin{theorem}
\label{morg}
For  $d \le 7$ let $S$ be an isoperimetric hypersurface
for $\mu = e^{-v} \ dx$. Assume that $v$ is $C^{k-1, \alpha}$, $k \ge 1, 0  < \alpha < 1$ and Lipschitz.
Then $S$ is locally a $C^{k, \alpha}$ manifold.
If $v$ is real analytic, then $S$ is real analytic.

For $d >7$ the statement holds up to a closed set of singularities with Hausdorff dimension
less than or equal to $d-7$.
\end{theorem}

In addition, we will use more special facts about log-concave radially symmetric distributions proved
in \cite{RCBM}.

\begin{itemize}
\item[1)] The balls about the origin are not isoperimetric (even stable) for strictly log-concave
radially symmetric distributions (Theorem 3.10).
\item[2)] The isoperimetric sets for strictly log-concave distributions have connected boundaries
(Corollary 3.9).
\end{itemize}

\section{Spherical symmetrization}

In this section we deal with the radially symmetric  measures.
We start with the case $d=2$. Denote by  $(r,\theta)$
the standard polar coordinate system.
Assume that $\mu=\rho(r) \ dx$ is supported on $B_{R}$, $R \in (0, \infty]$
and $\rho$ is smooth and positive on $B_{R}$.

{\bf Assumption:}  $A$ is an open set with Lipschitz boundary $\partial A$
(i.e. $\partial A$ is a finite union of graphs of Lipschitz functions).

We remark that according to Theorem \ref{morg} the  isoperimetric hypersurfaces satisfy this assumption.

\begin{definition}
We say that a set $A^*$ is obtained from $A \subset R^2$ by the circular symmetrization with respect to the
$x$-axis, if
for every $r \ge 0$ the set  $\partial B_r \cap A^*$ has the same length
as  $\partial B_r \cap A$ and, in addition,
 $\partial B_r \cap A^*$ has the form
 $\{-f(r) < \theta < f(r) \}$ for some $f \in [0,\pi]$.
 If $\partial B_r \subset A$, we require that $\partial B_r \cap A^* = \partial B_r$.
\end{definition}

\begin{remark}
By the Fubini theorem $A$ and $A^*$ have the same $\mu$-measure.
In addition,  the circular symmetrization can be defined with
respect to any ray starting from the origin.
\end{remark}

\begin{proposition}
The circular symmetrization does not increase the surface measure
$$
\mu^{+}(\partial A^*) \le \mu^{+}(\partial A).
$$

Assume that $A$ is connected,
$
\mu^{+}(\partial A^*) = \mu^{+}(\partial A),
$
and
$\mbox{\rm card}(\partial A \cap \partial B_r) < \infty$ for every $r>0$.
Then
$\mu(A^* \setminus U(A) ) =0$ for some rotation $U(r,\theta) = r e^{i(\theta+\theta_0)}.$
\end{proposition}
\begin{proof}

Without loss of generality we deal with a compact $A$ with $\mathcal{H}^{d-1}(\partial A) < \infty$.
It is known (see Theorem 3.42 in \cite{AFP}) that there exists a sequence
$A_n$ of smooth sets such that $I_{A_n} \to I_{A}$ almost everywhere (in the Lebesgue measure sense)  and
$\mu^{+}(\partial A_n) \to \mu^{+} (\partial A)$.
 Thus, to prove the first part of the Proposition, it is sufficient to deal
with sets with smooth boundaries.
We can even assume that every $\partial A_n$ is a level set of a polynomial
function $P_n$  restricted to a compact subset
($A_n$ obtained in the proof of Theorem 3.42 are level sets of smooth functions, we only apply the Weierstrass polynimila approximation theorem).
It is also possible to require that $\mbox{card}(A_n \cap \partial B_r) < \infty $.
Thus,  without loss of generality it is sufficient to consider the case when  $\partial A$
consists of finite number $n(r)$ of ordered Lipschitz curves $r \to r(\cos
f_{i}(r), \sin f_{i}(r))$, $r \to r(\cos g_{i}(r), \sin g_{i}(r))$
such that
$$
f_1 \le g_1 < f_2 \le g_2 < \cdots < f_n \le g_n
$$
and
$$
A = \{(r,\theta) : f_i (r) < \theta <  g_i (r)\}
$$  (we suppose  that $f_i = g_i$ just for a finite number of $r$). The further
proof  mimics the classical proof that Steiner's
symmetrization does not increase the perimeter. Indeed, for every
curve $r \to r(\cos \varphi(r), \sin \varphi(r))$ one has
$$
ds^2 = (1 + r^2 (\varphi')^2) dr^2.
$$
Hence
$$
\mu^{+}(\partial A) = \int_{0}^{\infty} \Bigl( \sum_{i=1}^{n(r)} \sqrt{(1+ r^2 (f'_i)^2)} +
\sqrt{(1+ r^2 (g'_i)^2} \Bigr)
\rho(r) \ I_r \ dr,
$$
$$
\mu^{+}(\partial A^*) = 2 \int_{0}^{\infty} \sqrt{1+ r^2 \Bigl( \sum_{i=1}^{n(r)} \frac{f'_i - g'_i}{2}\Bigr)^2}
\rho(r) \ I_r \ dr,
$$
where $I_r = \{ n(r) \ne 0 \}$ (all the $r$ such that $\partial B_r \cap \partial A \ne \emptyset$).
Note that the function $\sqrt{1 + x^2}$ is convex,
hence
$$
 \sum_{i=1}^{n} \Bigl( \sqrt{1+ a_i^2} +
\sqrt{1+ b_i^2} \Bigr)
\ge
2
 \sqrt{\Bigl(1 + \frac{1}{4}  \sum_{i=1}^{n}\bigl( a_i - b_i  \bigr)^2 \Bigr)}
$$
and we get the desired inequality.

Now assume that $\mu^+(\partial A) = \mu^+(\partial A^*)$, $\partial A$ is Lipschitz
and
$\mbox{\rm card}(\partial A \cap \partial B_r) < \infty$ for every $r>0$.
Then the above formulae hold. Clearly,
$\mu^{+}(\partial A) = \mu^{+}(\partial A^*)$ is possible only if
$n=1$, hence $A = \{(r,\theta): f(r) < \theta < g(r)\}$. Moreover,
  $f'=-g'$ for $r \in [r_1, r_2]$. Hence
 $f + g$ is constant on $[r_1, r_2]$. The equality is possible only if
 $n=1$ on some interval  $r_1 <  r < r_2$ (maybe unbounded) and $n=0$ outside. In this case
 $A^*$
 is obtained from $A$ by a rotation on the constant angle $\frac{f+g}{2}$
 (up to a zero measure set).
\end{proof}

\begin{corollary}
\label{d2sym}
Let $A$ be an isoperimetric set for a planar radially symmetric
density. Assume, in addition, that $A$ is connected, open, has an analytic boundary, and $\mathcal{H}^1(\partial B_r \cap \partial A) =0$ for every $r>0$. Then $A$ is stable under
the circular symmetrization with respect to some ray.
\end{corollary}
\begin{proof} It is sufficient to note that the analyticity of $\partial A$ and $\mathcal{H}^1(\partial B_r \cap \partial A) =0$
imply that $\mbox{\rm card}(\partial B_r \cap \partial A) < \infty$ and apply the previous Proposition.
\end{proof}

Analogously to the circular symmetrization let us introduce the spherical symmetrization on $\R^d$.

\begin{definition}
We say that a set $A^* \subset \R^d$ is obtained from $A \subset \R^d$ by the spherical
 symmetrization with respect to the
ray $R_a = \{ t a :  t \ge 0\}$ associated to a vector $a \in \R^d$, $a \ne 0$, if
for every $r \ge 0$ the set the $\partial B_r \cap A^*$ has the same Hausdorff $\mathcal{H}^{d-1}$-measure
as  $\partial B_r \cap A$ and, in addition,
 $\partial B_r \cap A^*$  is a spherical cap centered at $R_a \cap \partial B_r$.
\end{definition}

We don't prove here that the spherical symmetrization does not increase the surface measure (see \cite{MHH}).
Nevertheless, we show that every  isoperimetric set satisfying some additional technical assumptions
is stable under a spherical symmetrization. To this end let us introduce an intermediate operation.

\begin{definition}
Let $d \ge 3$ and $A$ be a Borel set.
 The set $A^{*}_{x_1,x_2}$ is determined by the following requirements.
 Fix coordinates $x_3, \cdots, x_d$. The intersection of every circle
 $$C_R = \{(x_1,x_2): x^2_1 + x^2_2 =R^2\}$$ with $A^{*}_{x_1,x_2}$ is an open arc $l_R \subset C_R$
 satisfying:
 \begin{itemize}
 \item[1)]
 $l_R$ has the same length as $l_R \cap A$ ($l_R = C_R$ if $C_R \subset A$ )
 \item[2)]
 the center $M=(x_1, x_2)$ of $l_R$ is uniquely determined by the requirement
 $x_1 \ge 0, x_2=0$.
 \end{itemize}
\end{definition}

In what follows we associate to any arbitrary  vector $a=(a^1, \cdots, a^d)$ the following
matrix $Q_a$:
$$
(Q_a)_{ij} =  a^i \cdot a^j.
$$

The following lemma follows from the convexity
of the function $x \to \det^{1/2} (I + Q_x)$.

\begin{lemma}
\label{matrix}
Let $M$ be a symmetric positive matrix. For a number of vectors $v_1, \cdots, v_{2n}$
the following inequality holds
$$
\sum_{k=1}^{2n} {\det}^{1/2} (M + Q_{v_k})
\ge  2 \cdot {\det}^{1/2} (M + Q_{v}),
$$
where $v = \frac{\sum_{k=1}^{n} v_k - \sum_{k=n+1}^{2n} v_k }{2}$.
In addition, an equality holds if and only if $n=1$
and $v_1=-v_2$.
\end{lemma}

\begin{proposition}
\label{symmetry}
Let $\mu$ be a radially symmetric measure.
Then $\mu^{+}(\partial A^*_{x_1 x_2}) \le \mu^{+}(\partial A)$.

Assume that every isoperimetric set of $\mu$ has analytic boundary. Let
$A$ be isoperimetric, connected, and $\mathcal{H}^1(C \cap \partial A)=0$
for every circle $$C = \{ x_1^2 + x^2_i = R^2; x_j = a_j\}, \ j \in \{2, \cdots, d\} \setminus \{i\} $$
with fixed $i \in \{2, \cdots, d\}$, $R>0$, $a_j \in \R$. Then there exists a ray $R_a = \{t \cdot a, t \ge 0\}$, $a \in R^d \setminus \{0\}$ such that every nonempty intersection of $A$ with any ball $B_r = \{x: |x| \le r\}$ is a spherical cap  (up to zero measure) with the center at $R_a \cap B_r$.
\end{proposition}
\begin{proof}
Let us show the first part. Without loss of generality we assume that $A$ is compact.
Consider the symmetrized set $A^*_{x_1
x_2}$. Clearly, $A^*_{x_1 x_2}$ has the same $\mu$-measure as $A$.
Let us show that $A^*_{x_1 x_2}$ has a smaller surface measure. Arguing in the same way as in the previous Proposition
we can assume that
$\partial A$ consists of finite number of smooth surfaces $S_{k}$
and
intersection of every $S_k$ with every
 circle $C = \{x^2_1 + x^2_2=R^2\}$ (other $x_i$ are fixed) consists of finite number of points.
Let us parametrize every surface $S_k$ in the following way:
$$
F_k : (r,  \tilde{x}) \to \bigl(\sqrt{r^2-{\tilde x}^2} \cos \theta_i(r, \tilde{x}),  \sqrt{r^2-\tilde{x}^2} \sin \theta_i(r, \tilde{x}), \tilde{x}\bigr),
$$
where $r$ is the distance from $F_k$ to the origin and  $\theta_k$ is the angle between the $O_{x_1}$-axis
and the projection of  $F_k$ onto $O_{x_1 x_2}$-plane, and $\tilde{x} = (x_3, \cdots, x_d)$.
In addition,
$$
A \cap C = \cup_{k=1}^{n} \{\theta_k < \theta < \theta_{n+k}\}.
$$
The first fundamental form $G_k$ of $S_k$ has the following
 representation in $(r, \tilde{x})$-coordinates:
 $$
 G_k=M+Q_k,
 $$
 where
 $$
 M_{rr} = \frac{r^2}{r^2-\tilde{x}^2}, \ M_{rx_i} = - \frac{r x_i}{r^2-\tilde{x}^2} , \ M_{x_i x_j} = \delta_{ij} + \frac{x_i x_j}{r^2-\tilde{x}^2}
 $$
 and
$$
Q_k = (r^2 - \tilde{x}^2) Q_{\nabla \theta_k},
$$
where $\nabla \theta_k = (\partial_r \theta_k, \partial_{x_3} \theta_k, \cdots, \partial_{x_d} \theta_k)$.
Hence
$
\mu^{+}(\partial A)
$
is equal to
$$
 \sum_{k=1}^{2n} \int_{0}^{\infty}  \int_{B_r}  {\det}^{1/2} \bigl( M + (r^2-\tilde{x}^2) Q_{\nabla \theta_k} \bigr) I(r, \tilde{x})  \ d\tilde{x} \ \rho(r) \ dr.
$$
Here
$
I(r, \tilde{x})
$
is the set of $(r, \tilde{x})$ such that $C$ has a non-empty intersection with
$A$.
Clearly,
$$
\mu^{+}(\partial A^*_{x_1 x_2})
= 2
 \int_{0}^{\infty}  \int_{B_r}
   {\det}^{1/2} \bigl( M + (r^2-\tilde{x}^2) Q_{\nabla \theta} \bigr)  I(r, \tilde{x}) \ d\tilde{x} \ \rho(r) \ dr,
$$
where
$$
\tilde{\theta} = \frac{\sum_{k=1}^{n} \theta_k - \sum_{k=n+1}^{2n} \theta_k }{2}.
$$
The desired inequality follows from Lemma \ref{matrix}.

Let us prove the second part. Take an isoperimetric set $A$ satisfying the assumptions. Note that the above formulae hold for $A$ and, in addition,
 $\mu(\partial A) = \mu(\partial A^*) $. This is possible  if and only if $n=1$   and $\nabla (\theta_2+\theta_1)=0$ on $r_1 \le r \le  r_2$
and $n=0$ for other values of $r$.
But this means that $A^*_{x_1 x_2}$ is obtained from
$A$ by a rotation.
Applying consequently  $x_1 x_i$-symmetrizations to $A$, we obtain a set $\tilde{A}$.
Since every $x_1 x_i$-symmetrization does not increase the surface measure,
the set $\tilde{A}$ is obtained  by a rotation of $A$ (up to measure zero).
 In addition, $\tilde{A}$ is symmetric with respect to any hyperplane $\pi_i = \{x_i=0\}$, $i>1$
  and $\partial \tilde{A}$ is connected.

 Now let us show that $\tilde{A}$ is symmetric with respect to any hyperplane $\pi$
 passing through the origin and containing $x_1$-axis. Indeed, since $\mu$ and
 $\tilde{A}$ are symmetric, the hyperplane $\pi$ divides $\tilde{A}$ in two pieces $A^+ \cup A^-$ with the same measure.
 Clearly, the Hsiang symmetrization $A^+ \cup s_{\pi}(A^+)$ ($s_{\pi}$ is the reflection with respect to $\pi$)
 is an isoperimetric set. Since the isoperimetric sets have smooth boundaries, it is possible if and only if
  $\partial \tilde{A}$ intersects  $\pi$ orthogonally. Hence  $A \cap \partial B_r$
 is a spherical cap  with the center at the $x_1$-axis for every $r>0$.
\end{proof}

\section{Stationary circular symmetric sets}

In this section we study the stationary sets of a radially symmetric measure
$\mu=e^{-v(r)} \ dx$ on the plane.

\begin{lemma}
\label{variat}
Assume that for some smooth $f$ one has:
\begin{equation}
\label{regular}
A = \{(r, \theta) : -f(r) < \theta < f(r) \}, \ f(r) \in [0,\pi]
\end{equation}
and $A$ is a stationary set.
Then $f$ satisfies
\begin{equation}
\label{u}
\dot{u} - \dot{v} u = c r,
\end{equation}
where $c$ is a
constant and
\begin{equation}
\label{f}
u=\frac{r^2 \dot{f}}{\sqrt{1 + r^2 (\dot{f})^2 } }.
\end{equation}
\end{lemma}
\begin{proof}
 One has
$$
\mu(A) = 2\int_{0}^{\infty} r f(r) \ \rho(r) dr,
$$
$$
\mu^{+}(\partial A) = 2\int_{0}^{\infty} \sqrt{1 + r^2 (f')^2} \
\rho(r)dr.
$$
Let us compute a variation of $\mu^{+}(\partial A)$
under the constraint $\mu(A)=C$. Consider an infinitesimal variation
$f + \varepsilon \varphi$
of $f$ by a smooth compactly supported function $\varphi$.
Since we keep $\mu(A)$ constant, we assume that  $\int \varphi r \rho \ dr=0$.
One obtains
$$
 \int_{0}^{\infty} \frac{r^2 f' \varphi'}{\sqrt{1 + r^2 (f')^2}} \ \rho(r) dr =0.
$$
Integrating by parts one gets
$$
 \int_{0}^{\infty}\Bigl(
 \bigl[ \frac{r^2 f' }{\sqrt{1 + r^2 (f')^2}}\bigr]' - v'(r)  \bigl[ \frac{r^2 f' }{\sqrt{1 + r^2 (f')^2}}\bigr]  \Bigr) \ \varphi \ \rho(r) \ dr =0.
$$
Taking into account that this holds for every smooth $\varphi$ with
$
\int \varphi r \rho \ dr =0,
$
one gets that $f(r)$ satisfies
$$
 \bigl[ \frac{r^2 f' }{\sqrt{1 + r^2 (f')^2}}\bigr]' - v'(r)  \bigl[ \frac{r^2 f'}{\sqrt{1 + r^2 (f')^2}}\bigr] = c \cdot r
$$
for some constant $c$.
The proof is complete.
\end{proof}

\begin{remark}
Clearly,  Lemma \ref{variat} can be generalized to higher dimensions.
Let $d=3$ and $\partial A$ is parametrized in the following way:
$$
\left\{
\begin{array}{lcr}
x = r \sin f(r) \cos \varphi \\
y = r \sin f(r) \sin \varphi \\
z = r \cos f(r). \\
\end{array}
\right.
$$
Then
$$
\mu(A) = 4 \pi \int_{0}^{\infty} r^2  (1-\cos f) \rho(r) dr,
$$
$$
\mu^{+}(\partial A) = 2\pi \int_{0}^{\infty}  r \sin f \sqrt{1 + r^2 (f')^2} \ \rho(r)dr.
$$
Arguing as above, we obtain
$$
 \Bigl( \frac{r^3 f'}{\sqrt{1+ (rf')^2}} \Bigr)' - v' \frac{r^3
f'}{\sqrt{1+ (rf')^2}} =  \frac{r}{\sqrt{1+ (rf')^2}} \mbox{ctg} f
- c r^2.
$$
\end{remark}

\begin{remark}
We remark that equations (\ref{u})-(\ref{f}) follows also from a result of
\cite{RCBM}: an isoperimetric surface $S$ with density $e^{-V}$ has  a constant generalized
mean curvature
$$
(d-1)H -  \langle n, \nabla V \rangle,
$$
where $H$ is the Euclidean mean curvature of $S$ and $n$ is the normal vector of $S$.

Note, however, that the spheres about the origin  always have constant generalized
mean curvature for every radially symmetric density. The ball $B_{r_0}$ corresponds to the singular function
$f(r) = 2\pi \chi_{[0, r_0]}$ and $r$ is tangent to $u$ at $r_0$.
In addition, the halfspace $H_{v} = \{x: \langle x, v \rangle \ge 0\}$ through the
origin gives another example of a surface of  a constant generalized mean curvature. This set corresponds to
the constant solution $f=\frac{\pi}{2}$.
\end{remark}

\begin{example}
Consider the standard
Gaussian measure  ($v=\frac{r^2}{2}$).
Among all the solutions to (\ref{u}) take the ones growing not faster than
a linear function. These are the constants.
For
$u=c$ solve  (\ref{f}). The solution
$$
f=\arccos\Bigl(\frac{r_0}{r}\Bigr), \ \ r \ge r_0
$$
defines an isoperimetric surface  (a halfspace).
\end{example}

\begin{lemma}
\label{analytic}
Assume that $v$ is real analytic, $A$ is a connected isoperimetric set
  and $\partial A$ is  connected. Then
$A$ is either a ball
or a circularly symmetric set (with respect to some ray).

If $A$ is circularly symmetric with respect to the $x$-axis, then
$$
\partial A \cap \{ y \ge 0\}\subset  \{r e^{if(r)}, \ r \in [r_0,r_1]\},
$$
where $f$ is a solution to (\ref{u})-(\ref{f}) and $[r_0,r_1]$ is the maximal interval of the existence
of the solution to  (\ref{u})-(\ref{f}).
\end{lemma}
\begin{proof}
According to the regularity results (see Section 2), $\partial A$ is real analytic.
If $\mathcal{H}^{1}(\partial A \cap \partial B_r) >0$, then $\partial A$ contains an arc and by the uniqueness
of analytic continuation $A$ is a ball. If  $\mathcal{H}^{1}(\partial A \cap \partial B_r) = 0$, then
$A$ is symmetric with respect to some circular symmetrization by Corollary  \ref{d2sym}.

Assume that $A$ is circularly symmetric with respect to the $x$-axis. The boundary $\partial B_{R} \cap \partial A$  contains exactly two points for every  $R>0$ from an open interval $(a,b)$ and
(\ref{regular}) holds in a neighborhood of $R$ with $f$ solving (\ref{u})-(\ref{f}). Let $[r_0,r_1]$
be the interval of existence to (\ref{u})-(\ref{f}).
Assume that  $l: r \to  e^{if(r)}$, $r \in [r_0,r_1]$ does not
cover  the intersection of $\partial A$ with the halfplane $\{y \ge 0\}$.
Clearly, in this case $l([r_0, r_1])$ is compact and $f'(r_1)=\infty$.
There exists a unique analytic continuation $l_{\delta}$ of this curve for $r \in [r_1, r_1 + \delta)$ for
some $\delta>0$. This continuation does satisfy
(\ref{u})-(\ref{f}) as well (note that by the uniqueness of the continuation
 $l_{\delta}$ can not be an arc). Since $v$ is radially symmetric,  by the uniqueness of the solution to an ODE with given initial data, the curve
$l_{\delta}$ coincides with the reflection of
$l({(r_1-\delta',r_1]})$ with respect to the line $\theta =
f(r_1)$. But this clearly contradicts to the fact that $A
\cap \partial B_r$ is an arc for every $r>0$. Hence $\partial A \subset \{r e^{if(r)}, \ r \in [r_0, r_1] \}$.
\end{proof}

Let us give some examples when (\ref{u})-(\ref{f}) is explicitly solvable.

\begin{example}
Consider the Lebesgue measure  ($v=0$).
Let us find the solution to (\ref{u})-(\ref{f}) such that $u(r_0)=r_0$, $u(r_1)=r_1$.
One easily obtains
$$
u= \frac{r^2 + r_0 r_1}{r_0 + r_1}, \ r_0 + r_1 \ne 0.
$$
The solutions are the balls having the segment $[r_0,r_1]$ as the intersection with the $x$-axis.
The formula makes sense also for negative values of $r_0, r_1$.
Note that the case $r_0 + r_1=0$ corresponds to a constant $u$, hence to the balls about the origin.
In addition, infinite values of $r_0, r_1$ correspond to the half-spaces (which are also stationary).
\end{example}

\begin{example}
{\bf Stationary symmetric sets for $\mu=\frac{dx}{r}$}.

In this case the isoperimetric sets do not exist (see \cite{CJQW}).
We show by solving explicitly
 (\ref{u})-(\ref{f}) that the only  stationary sets with smooth connected boundaries are
  the balls about the origin and halfplanes passing through the origin.

Indeed, consider a  stationary set which is not a ball and not
a halfspace ($u \ne 0$). Solving $ \dot{u} - \frac{1}{r} u = ar $
we get $u = ar^2 + \lambda r$. Note that the solution to (\ref{f})
exists for $r$ satisfying $|u(r)| \le r$, thus
$$
|\lambda + ar| \le 1.
$$
By the symmetry arguments it is enough to consider the case $\lambda \ge 0$.
Assume first that $\lambda >0$.
Note that for $\lambda \le 1$ one has
$$
f(r) = \int_{0}^{r} \frac{as + \lambda}{s \sqrt{1-(\lambda + as)^2}} \ ds.
$$
But this integral diverges. Hence, it makes only sense to consider  the case
$$
\lambda >1, \ a<0.
$$
One obtains
$$
f(r) = \int_{\frac{1-\lambda}{a}}^{r} \frac{as + \lambda}{s \sqrt{1-(\lambda + as)^2}} \ ds, \ \ \
\frac{\lambda-1}{-a} \le r \le   \frac{1+\lambda}{-a}.
$$
Computing this integral we get
$$
a \int_{\frac{1-\lambda}{a}}^{r} \frac{1}{\sqrt{1-(\lambda + as)^2}} \ ds
= \arcsin (\lambda + ar) - \frac{\pi}{2} ,
$$
$$
 \lambda\int_{\frac{1-\lambda}{a}}^{r} \frac{ ds}{s \sqrt{1-(\lambda + as)^2}}
 =
 \lambda\int_{\frac{a}{1-\lambda}}^{\frac{1}{r}} \frac{dt}{ \sqrt{t^2-(t\lambda + a)^2}}.
$$
The latter integral is equal  to
$$
\frac{\lambda}{\sqrt{\lambda^2-1}} \Bigl( \frac{\pi}{2} +
\arcsin \Bigl( \frac{\lambda^2 + \lambda a r -1}{ar} \Bigr)  \Bigr).
$$
Finally, we obtain
$$
f(r)
= \arcsin (\lambda + ar) - \frac{\pi}{2}
 + \frac{\lambda}{\sqrt{\lambda^2-1}} \Bigl( \frac{\pi}{2} +
\arcsin \Bigl( \frac{\lambda^2 + \lambda a r -1}{ar} \Bigr)  \Bigr).
$$
One easily verify that this curve is non-closed. Indeed, if $r =r_1 = -\frac{1+\lambda}{a}$, one has
$$
f(r_1) = \pi \Bigl(1 -\frac{\lambda}{\sqrt{\lambda^2-1}}\Bigr)
$$
and this is neither $0$ nor $ \pi$.

Solving (\ref{u})-(\ref{f}) for $\lambda=0$  and taking into account that
the solution should be smooth for $r \ne 0$  we obtain a family of circles containing the origin
$$
x^2 + y^2 = \frac{x}{a}.
$$
It is easy to check that these circles have infinite length.
\end{example}

\begin{remark}
Some very interesting results on isoperimetric sets for measures $\mu = r^p \ dx$
have been recently obtained in \cite{DDNT}.
\end{remark}

\section{Computations of isoperimetric sets for exponential power laws}

In this section we compute numerically the isoperimetric sets of the measures
$$\mu_{\alpha} = C_{\alpha} e^{-r^{\alpha}} \ dx, \ \alpha \ge 1.$$
The isoperimetric estimates for these kind of laws have been obtained in
\cite{Huet}.

 Note that by the results mentioned in Section 2:
 \begin{itemize}
\item[1)]
 The isoperimetric sets do exist and have at least $C^{1,\varepsilon}$ boundary
for some $\varepsilon >0$. Note that the origin is the only point where the potential is not analytic.
Thus, the isoperimetric curve is analytic at any other point.
\item[2)]
The boundaries of isoperimetric sets are connected.
 \item[3)] Balls about the origin are NOT isoperimetric for $\alpha>1$.
\end{itemize}

Solving equation (\ref{u}) we get
\begin{equation}
\label{u-a}
u = a e^{r^{\alpha}} \int_{r}^{\infty} s e^{-s^{\alpha}} \ ds +
\lambda e^{r^{\alpha}}.
\end{equation}

For $\lambda=0$, $u$ has a growth of the order $r^{2-\alpha}$.
 In any other case $u$ has a growth of the order $e^{r^{\alpha}}$.

Equation (\ref{f}) is equivalent to the equation
$$
\dot{f} = \frac{u}{r\sqrt{r^2-u^2}}.
$$

Let $l = r e^{i f}$ be a curve solving (\ref{u})-(\ref{f}).
Without loss of generality we may assume that $a\geq 0$, because
for the opposite values of $a$ and $\lambda$ one obtains the curve which
is symmetric to $l$ with respect to the $y$-reflection. Note that
$|u(r)|<r$ for any $r$ in the interval of existence for $l$. For
$\alpha=1$ and $\lambda=0$ one has $u(r)=a(1+r)$. For $\alpha=2$
and $\lambda=0$ one has $u(r)=a$. Thus for $\alpha=1$, $\lambda=0$
and $0\leq a<1$ the interval of existence of $l$ is
$(r_0,+\infty)$ for some $r_0>0$ and is empty for $a\geq 1$. For
$\alpha=2$, $\lambda=0$ the interval of existence of $l$ is
$(r_0,+\infty)$ for some $r_0>0$.  We find empirically that for
$\lambda=0$ and $\alpha\geq 1$ the equation $u(r)=r$ has no more
than one solution. This implies that the interval of existence of $l$ is infinite.
If $\lambda\ne0$ then the growth rate of $u(r)$ is higher than $1$ and
the existence interval of $l$ is either empty or finite.

 Let us describe different types of
behavior of $l$.

\begin{itemize}
\item[1)]
The curve is non-compact. This corresponds to the case when
$\lambda=0$.
 Indeed,
otherwise equation $r^2 - u^2 \ge 0$ is satisfied on a compact
interval and $l$ is compact.

Assume that the curve does not touch the origin and $r_0$ is the
smallest value of $r$ such that $B_r$ and $l$ has a non-empty
intersection. The angle of rotation of the curve when $r$ changes
from $r_0$ to $r$ is equal to
\begin{equation}
\label{fu} f(r)= \int_{r_0}^{r} \frac{u}{s \sqrt{s^2-u^2}} \ ds.
\end{equation}

The full rotation of $l$ is equal to $\Delta_f= \int_{r_0}^{\infty}
\frac{u}{s \sqrt{s^2-u^2}} \ ds.$ In the case of compact curve the
full rotation is equal to $\Delta_f= \int_{r_0}^{r_1} \frac{u}{s
\sqrt{s^2-u^2}} \ ds,$ where $[r_0,r_1]$ is the largest existence
interval for $l$.

The  curves can be self-intersecting or non self-intersecting.

\includegraphics[scale=0.6]{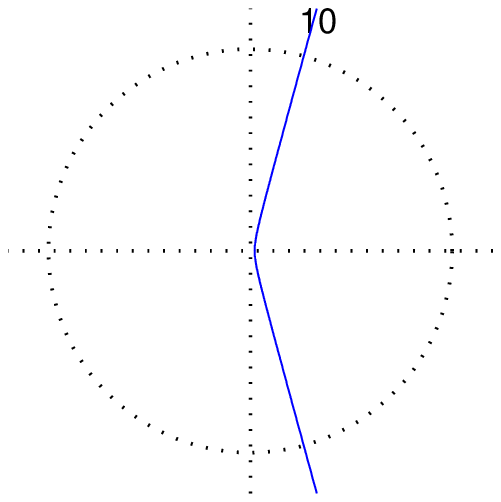}
\qquad
\includegraphics[scale=0.6]{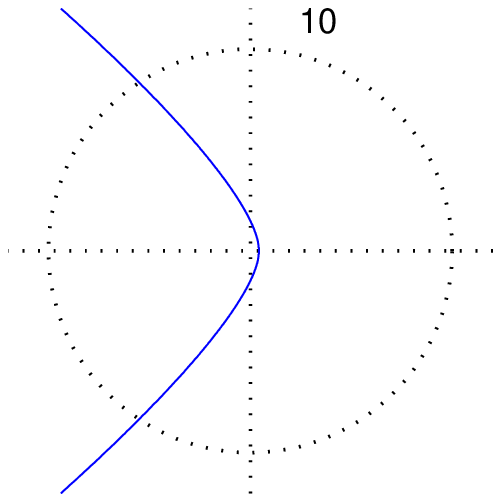}

{\bf Figure 1:} { Non-compact and non self-intersecting curves}. Smooth non-compact solutions to (\ref{u})-(\ref{f})
in the super-Gaussian case ($\alpha=3$, $a=0.5$) and sub-Gaussian case ($\alpha=1.3$, $a=0.5$).

If the full rotation exceeds $\pi$ then  two branches of $l$ have
an intersection. Obviously, in this case the curve can not be a
boundary of an isoperimetric set. Clearly, since isoperimetric
sets have smooth boundary, the part of a curve between $r_0$ and
the point of intersection  $r_1$ can be isoperimetric only if $\dot f(r_1)=\infty$.

It was realized by numerical computations that for $\alpha \geq 2$
 the full rotation is less then $\pi$ and $l$ is not
self-intersecting. If $\alpha=1$ then for all $a$ the curve $l$
is self-intersecting. It is easy to verify that for
$\alpha=1$ the full rotation is infinite and for $\alpha>1$ the
full rotation is finite.

\item[2)]

The curve intersects itself non-smoothly. This happens only for $1 \le \alpha < 2$.
Clearly, in this case the curves are not isoperimetric.

\includegraphics[scale=0.6]{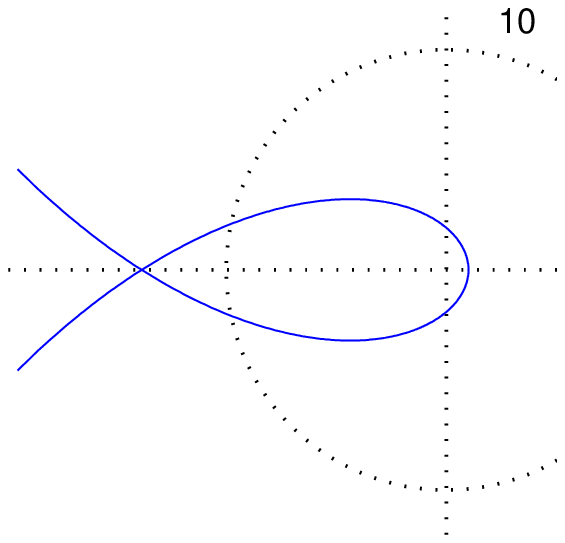}
\includegraphics[scale=0.6]{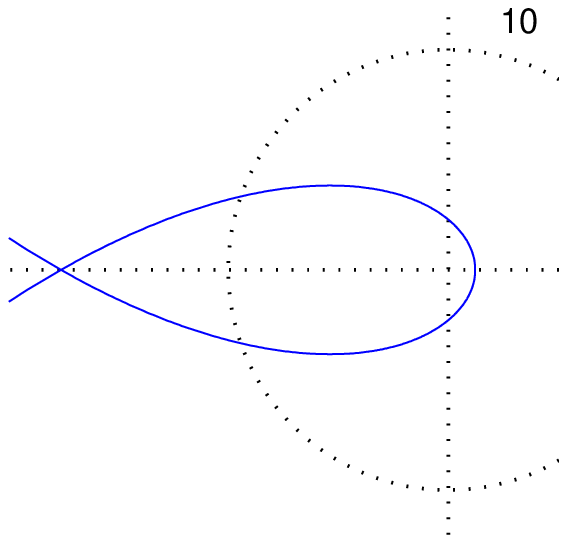}

{\bf Figure 2:} Self-intersecting curves. Solutions to (\ref{u})-(\ref{f}) in the  case
$\alpha=1$, $a=0.5$ and $\alpha=1.1$, $a=0.7$.

\item[3)]
The curve is compact but not closed ($f'_r =\infty$) for
$|f|<\pi$. This type of behavior occurs for any $\alpha \ge 1$.
These curves can not be isoperimetric since they have no analytic
circular symmetric continuation (see Lemma \ref{analytic}).

\includegraphics[scale=0.6]{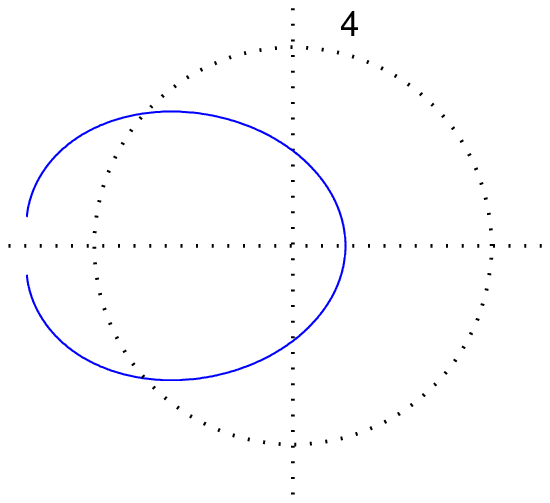}
\includegraphics[scale=0.6]{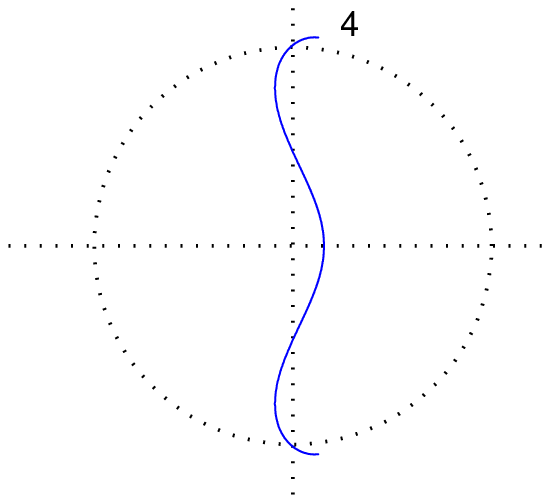}

{\bf Figure 3:} Non-closed  solutions to (\ref{u})-(\ref{f}) in the  case $\alpha=1$ ($a=0.5, \lambda=0.01$ and $a=0.5, \lambda=-0.1$).

Clearly, in this case $
0< \Delta_f < \pi.
$

 \item[4)]
The curve is closed and smooth ($f'_r=\infty$ if
 $f=\pm \pi$). These types of curves do appear for any
 $1 \le \alpha <2$.

\includegraphics[scale=0.6]{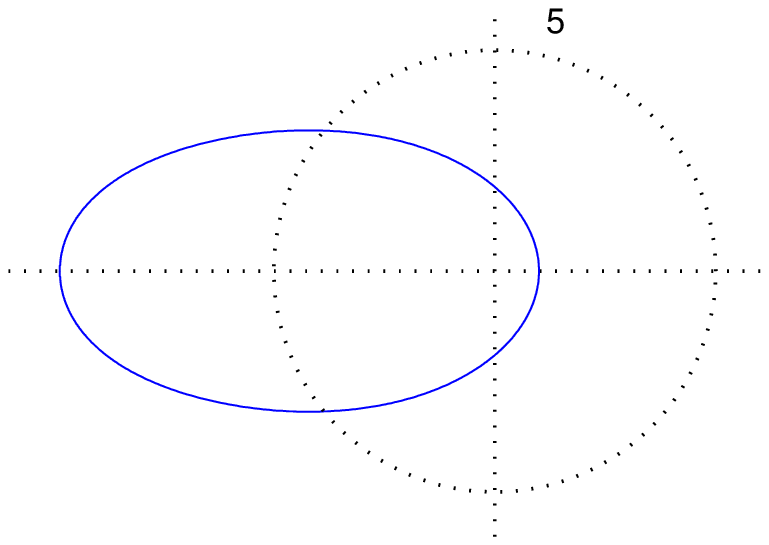}

{\bf Figure 4:} Closed smooth solution to (\ref{u})-(\ref{f}) in the  case $\alpha=1$, $a=0.5$.

\item[5)]
Special case: curve starting from the origin. The main difference
to previous cases: for smooth curves one has $f(0)=\pi/2$ (unlike
$f(r_0)=0$ for $r_0>0$).

\includegraphics[scale=0.6]{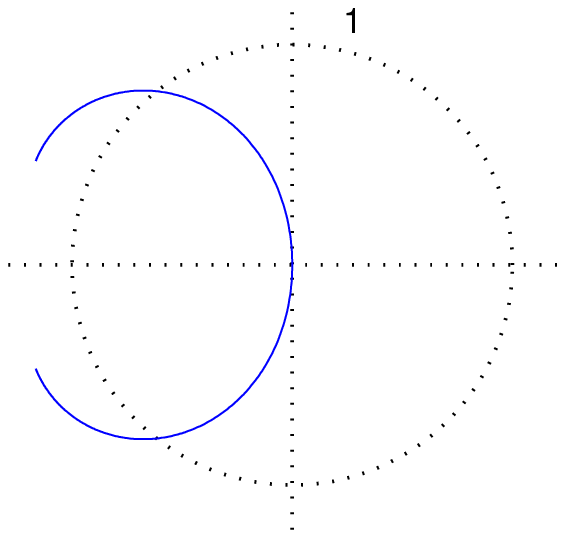}

{\bf Figure 5:} Curve starting from the origin $\alpha=1$.

\end{itemize}

We compute numerically the isoperimetric sets for the super-Gaussian ($\alpha >2$), sub-Gaussian ($1<\alpha<2$), and exponential
$\alpha=1$ distributions.
We stress  that the results below are partially justified by numerical computations.

\centerline{\bf 1) Super-Gaussian case, $\alpha >2$.}

In this case the balls around the origin are not isoperimetric.
For every $a$ and $\lambda=0$  the solutions to
(\ref{u})-(\ref{f}) are non-compact and non self-intersecting (case 1)).
In this case
\begin{equation}
\label{u-alpha}
u =  a e^{r^{\alpha}} \int_{r}^{\infty} s e^{-s^{\alpha}} \ ds
\end{equation}

It was verified numerically that  all the compact curves ($\lambda \ne 0$) solving
(\ref{u})-(\ref{f}) correspond to the case 3) but not to 2) or 4). The same  happens to
curves starting from the origin. Hence, the compact curves are not isoperimetric.

{\bf Conclusion:} For any given value of measure there exists a
unique (up to a rotation) open isoperimetric set $A$ and a unique
parameter $a$ such that $\partial A = \{ re^{if(r)} \}$ with $u$ given by
(\ref{u-alpha}) and $f$ given by (\ref{fu}). The set $A $ is one
of the sets obtained by dividing the plane by the curve $l$ in two
pieces.

\includegraphics[scale=0.7]{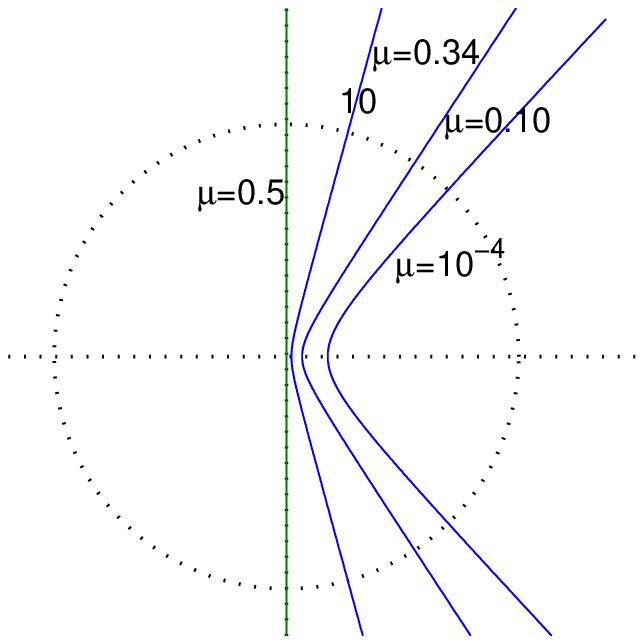}

{\bf Figure 6:} Experimental computation of the isoperimetric sets for $\alpha=3$.
The number $m$ is equal to the  value of $\mu$ of the corresponding {\bf convex} region. The green line ($A$ is a halfspace) corresponds to the case $\mu(A)=1/2$.

\centerline{\bf 2) Exponential case, $\alpha =1$.}

In this case the balls about the origin can not be excluded as eventual isoperimetric sets.

One has
\begin{equation}
\label{expon}
u=a(r+1) + \lambda e^r.
\end{equation}

It is easy to see that case 1) is not possible because
 (\ref{fu}) diverges as $r$ goes to infinity.

 In case 5) it can be verified numerically that the full rotation of the curves is less than $\pi/2$ (in this case $\theta$ starts
from $\pi/2$)
(see Figure 5.).

Now fix the parameter $a$ and change $\lambda$.
For positive values of $\lambda$ the full rotation of the curve depends
monotonically on $\lambda$. The same holds for  $-a<\lambda<0$. It turns out that for negative values of $\lambda$
the curves are either non-closed or
intersect themselves non-smoothly. This observation allows to describe
the isoperimetric sets. Indeed, there exists a unique positive value of
$\lambda$ such that both ends of the curve meet smoothly.
In addition, there exists $\lambda<0$ such that both ends of the curve meet smoothly
but in this case the  curve is self-intersecting.

\includegraphics[scale=0.8]{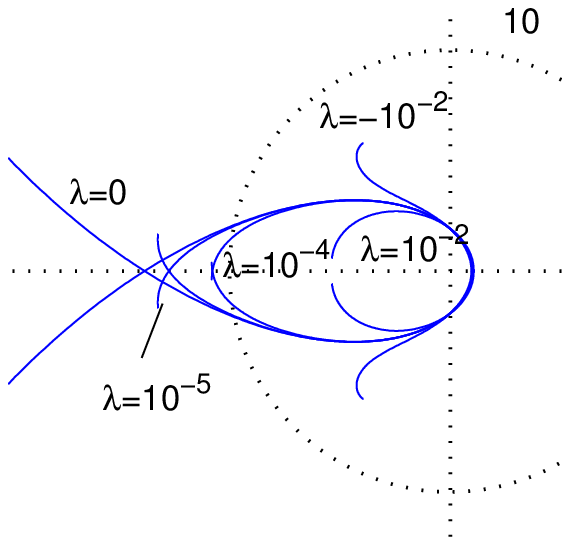}
\includegraphics[scale=0.8]{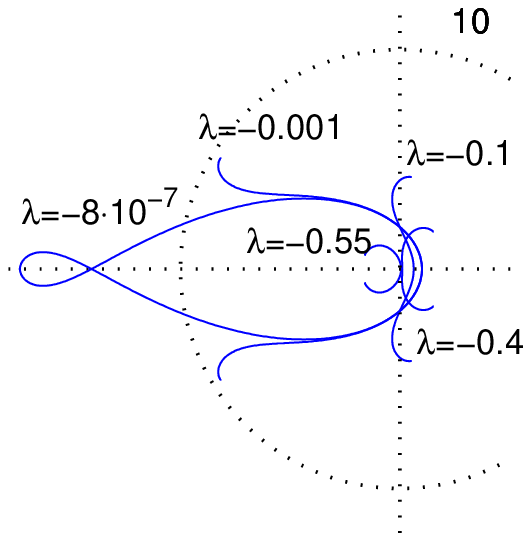}

{\bf Figure 7:} Solutions to (\ref{u})-(\ref{f}) in the  case
$\alpha=1$, $a=0.5$ and different values of $\lambda$.
The smooth curve appears for $\lambda \approx 10^{-4}.$

Thus we conclude that {\it every $a$ can have only one $\lambda$
 corresponding to a smooth  curve.}

\includegraphics[scale=0.6]{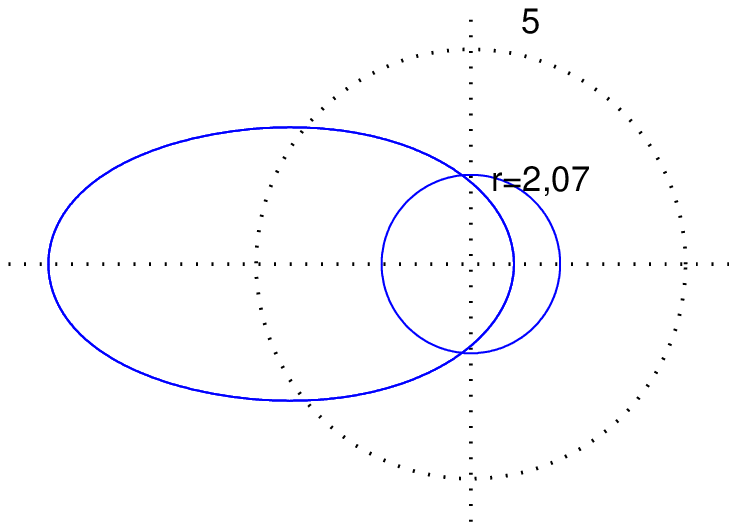}
\includegraphics[scale=0.6]{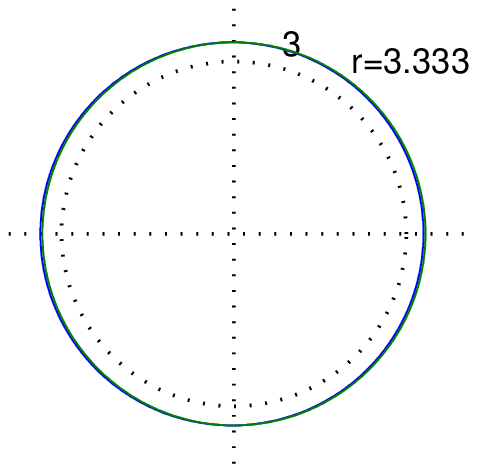}

{\bf Figure 8:} Isoperimetric sets in the  case $\alpha=1$,
$a=0.5$ and $a=0.7$ and the corresponding balls of the same
measure (they almost coincide on the picture).

We managed to find a family of smooth stationary
curves solving (\ref{u})-(\ref{f}). They do not touch the origin. According to Lemma \ref{analytic} they have only the circles around the origin as competitors. The
computations show: circles are always worse!

\includegraphics[scale=0.8]{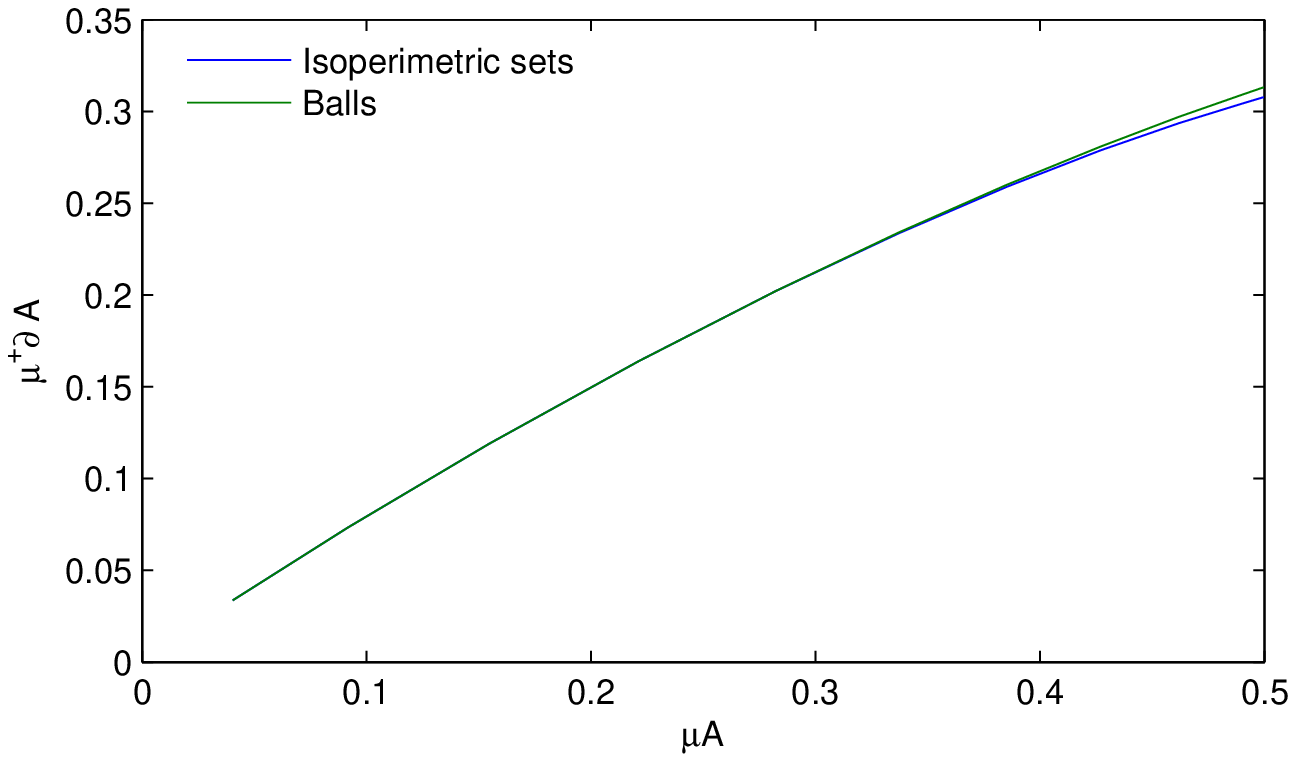}

{\bf Figure 9:} Isoperimetric function and dependence of the surface measure from the measure for the balls.

\

{\bf Conclusion:} For any given value of measure there exists a unique (up to a rotation)
isoperimetric curve $l = \partial A$ defined by a unique couple of parameters $a, \lambda$ such that
the corresponding isoperimetric set is either the compact convex set inside of $l$
or its complement. In this case
$\partial A =  \{ re^{if(r)} \}$,  $u$ is given by (\ref{expon}), and $f$ given by (\ref{fu}).

\includegraphics[scale=0.6]{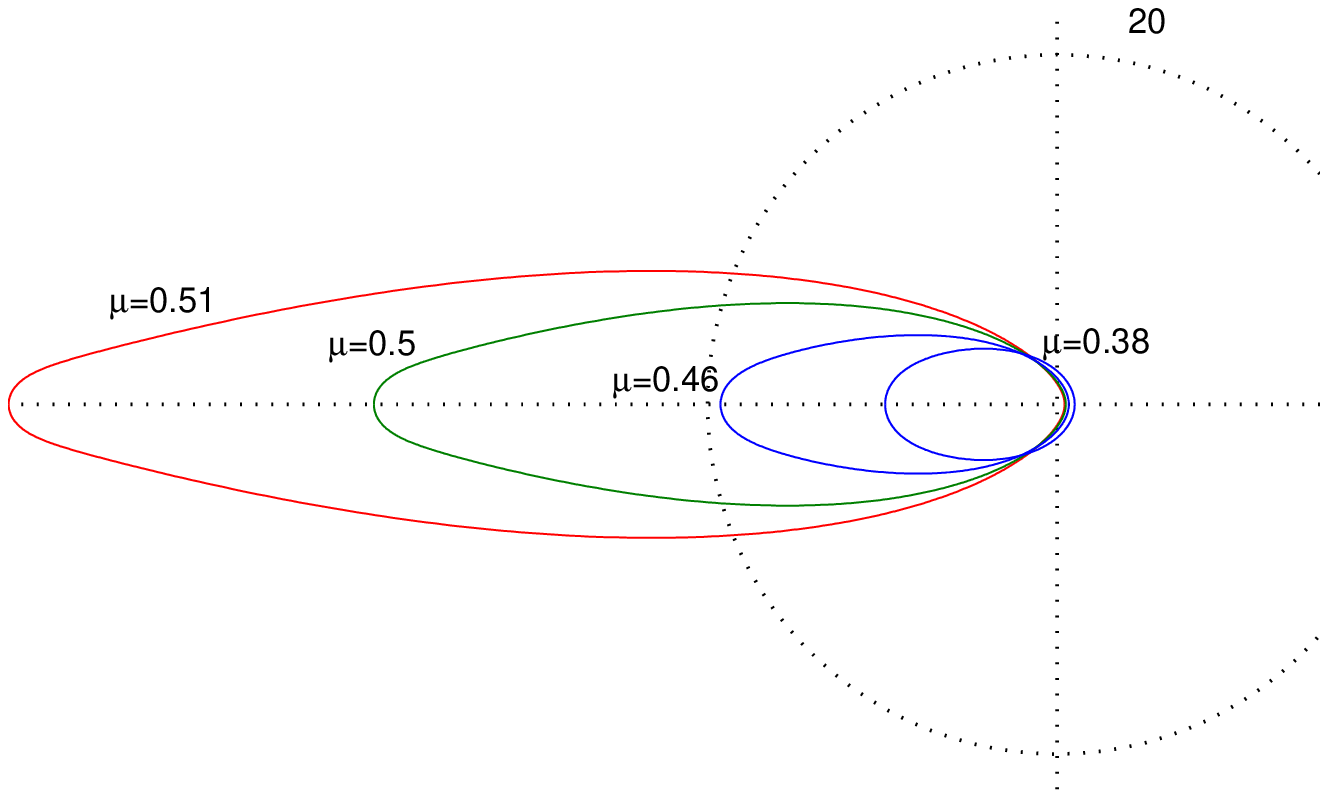}

{\bf Figure 10:} Experimental computation of the isoperimetric sets for $\alpha=1$.
The number $m$ is equal to the  value of $\mu$ of the corresponding {\bf non-convex} region. The green line corresponds to the case $\mu(A)=1/2$.  The red line ($m=0.51$) corresponds to a stationary smooth curve which is not isoperimetric (the best choice can be found among of the family of blue curves with $m=0.49$). For small values of $m$ the curves are asymptotically round in shape.

\centerline{\bf 3) Sub-Gaussian case, $1 < \alpha <2$.}

In this case there exist compact as well non-compact smooth stationary curves.
However, some of them are not isoperimetric (similarly to the red curve from the Figure 10).
It was verified numerically that there exist critical values $1< \alpha_0 < \alpha_1 \le 2$
such that the isoperimetric sets are compact for $1 < \alpha < \alpha_0$ ($\alpha_0 \approx 1.08$) and non-compact for
$\alpha_1 < \alpha$. Obviously, $\alpha_1 \le 2$. Some heuristic arguments demonstrate that $\alpha_1$ can be equal
to $1.5$. Indeed, we know that $u \sim a r^{2-\alpha}$. Applying formula for $\Delta_f$ with $u = ar^{2-\alpha}$
we get that the full rotation $\Delta_f < \pi$ for $\alpha> 1.5$.

{\bf Conclusion:}
\begin{itemize}
\item[1)] there exists $a_0 = a_0(\alpha)$ such that for $a = a_0$ the corresponding curve given by (\ref{u-a}) is the boundary of an isoperimetric set of measure $0.5$;
\item[2)] for $a < a_0$ the curve given by (\ref{u-a}) is not isoperimetric;
\item[3)] for $a > a_0$ the curve is isoperimetric and the measure of the part of the plane that does not contain the origin is less than $0.5$;
\item[4)] there exists $a_1 = a_1(\alpha)$ such that for every $a < a_1$ the curve given by (\ref{u-a}) is non-compact (this corresponds to the case $\lambda = 0$) and the full rotation is less than $\pi$
(see Figure 1);
\item[5)] for every $a > a_1$  there exists a unique $\lambda$ such that the isoperimetric curve
 is compact, closed and smooth (analogously to the case $\alpha = 1$);
\item[6)] for $1<\alpha< \alpha_0$ one has  $a_0>a_1$. Thus for $\alpha<\alpha_0$ all the isoperimetric curves are compact (exponential type);
\item[7)] for $\alpha_1>\alpha>\alpha_0$ the critical value $a_0$ is less than $a_1$. Thus for $\alpha_1>\alpha>\alpha_0$ there exist compact isoperimetric curves as well as non-compact isoperimetric curves.
\item[8)] for $\alpha>\alpha_1$ and for every $a>0$ the full rotation of corresponding curve given by (\ref{u-a}) is less then $\pi$. Thus every isoperimetric curve is non-compact (super-Gaussian type).
\end{itemize}

{\bf Figures 11-13:} Experimental computation of the isoperimetric sets for $\alpha=1.1; 1.2; 1.5$.
The number $m$ is equal to the  value of $\mu$ of the corresponding {\bf non-convex} region. The green line corresponds to the case $\mu(A)=1/2$.  The red line corresponds to a stationary smooth curve which is not isoperimetric.

\includegraphics[scale=0.7]{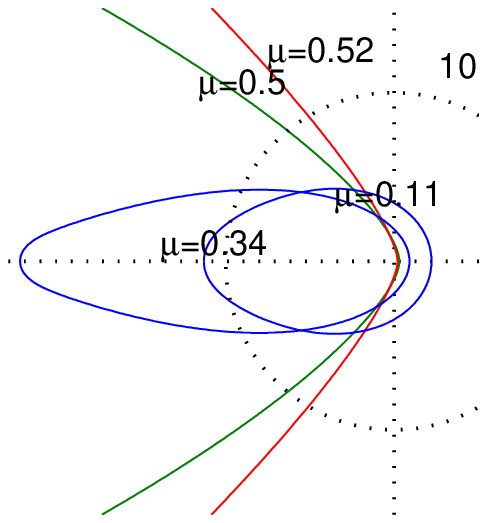}
\includegraphics[scale=0.7]{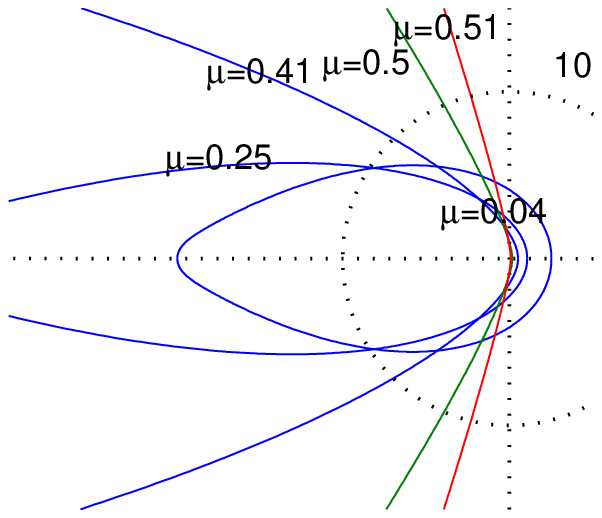}

\includegraphics[scale=0.7]{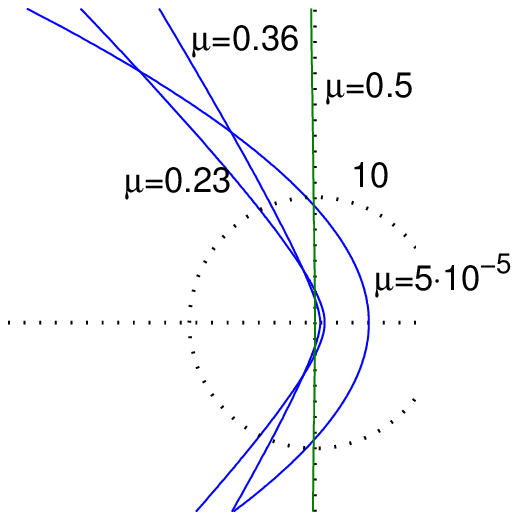}

\section{Log-convex measures}

In this section we investigate the so called log-convex measures,
i.e. measures of the form $\mu = e^{V} \ dx$, where $V$ is convex.
The measures of this type can be considered as natural analogs of
the negatively curved spaces in geometry.

Recall that the hyperbolic space $H^{d-1}$
serves as a model space for the negatively curved spaces. Solutions to the isoperimetric problem are given
by the metric balls. The hyperbolic plane  $H^2$
enjoys the following isoperimetric inequality:
\begin{equation}
\label{hd}
4 \pi \nu(A) -  K \nu^2(A) \le  [\nu^{+}(\partial A)]^2,
\end{equation}
where $\nu$ is the Riemannian volume on
$H^2$, $\nu^{+}(\partial A)$ is the length of the boundary $\partial A$, $\nu(A) < \infty$  and
$K<0$ is the constant Gauss curvature.

 Analyzing (\ref{hd}) and Borell's result \cite{Bor-iso2} one can conclude that a natural isoperimetric inequality
 for a log-convex measure has the form
 \begin{equation}
 \label{log-conv-iso}
  \mu(A)^{1-\frac{1}{d}} +  \mu(A) \psi(\mu(A)) \le C(d)
\mu^{+}(\partial A)
\end{equation}
for some increasing $\psi$. Here the first term in the left-hand
side is "responsible"  for small values of $\mu(A)$ and the second
one for large one's.

The following conjecture was suggested in  \cite{RCBM}.
\begin{conjecture}
\label{lcr-conj} Let $\mu = e^{v(r)} dx $ be a radially symmetric
measure with a convex smooth potential $v$ on $R^d$, $d \ge 2$. Then the
isoperimetric regions for $\mu$ are the balls about the
origin $B_R = \{x: |x| \le R \}$.
\end{conjecture}

Here we prove some particular cases of this conjecture and some related results.

\subsection{Divergence theorem and radially symmetric measures}

We start with an elementary lemma based on the divergence theorem. This
lemma allows to describe asymptotically the isoperimetric function
for radially symmetric measures. We deal below with open sets with Lipschitz boundaries.

\begin{lemma}
\label{ibp-estim} For $\mu=e^{V} dx$ with a sufficiently regular $V$ one has
$$
\mu^{+}(\partial A)
\ge
\int_{A} \Bigl( \mbox{\rm div} \frac{\nabla V}{ |\nabla V|}
+ |\nabla V|
\Bigr) d\mu.
$$
In particular, if $V = v(r)$, then
$$
\mu^{+}(\partial A)
\ge
\int_{A} \Bigl(  \frac{d-1}{ r}
+ v'(r)
\Bigr) d\mu.
$$
If $V$ has convex sublevel sets, then
$$
\mu^{+}(\partial A)
\ge
\int_{A}   |\nabla V| d\mu.
$$
\end{lemma}
\begin{proof}
The result follows from the trivial inequality
$$
\mu^{+}(\partial A)
\ge
\int_{\partial A} \Bigl\langle n_A, \frac{\nabla V}{|\nabla V|} \Bigr\rangle e^V d\mathcal{H}^{d-1}
$$
and integration by parts.
\end{proof}

\begin{example}
\label{exp-meas}
Measures $\nu_2 = e^{r} dx$ and $\nu_1 = e^{\sum_{i=1}^{d} |x_i|} dx$
satisfy the following Cheeger-type  inequalities:
$$
\nu^{+}_2(\partial A) \ge \nu_2(A),
$$
$$
 \nu^{+}_1(\partial A)\ge \sqrt{d} \ \nu_1(A).
$$
\end{example}

Let us apply this result in the radially symmetric case.

\begin{lemma}
\label{rad-sym-lemma}
Let $\mu$ be any Borel positive measure on $\R^d$ and $F : \R^d \to \R^{+}$
be any function such that the corresponding distribution function $\mu_{F}(t) =
\mu (C_t)$, where $C_t := \{x: F(x) \le t\}$, is continuous and strictly increasing
from $0$ to $\mu(\R^d)$. Then for every Borel set $A$ with finite measure one has
$$
\int_{A} F d \mu \ge \int_{C_t} F d \mu,
$$
where $t$ is chosen in such a way that $\mu(A) = \mu(C_t)$.
\end{lemma}
\begin{proof}
First we note that the existence of $t$ satisfying $\mu(A) = \mu(C_t)$ follows from the assumptions.
Next
\begin{align*}
\int_{A} F(x) d\mu
& =
\int_{A \cap C_t} F(x) d\mu
+ \int_{A \cap C^c_t} F(x) d\mu
\\&
\ge \int_{A \cap C_t} F(x) d\mu + \int_{A \cap C^c_t} t d\mu
\\& = \int_{A \cap C_t} F(x) d\mu
+ t \mu(A^c \cap C_t) \ge \int_{C_t} F(x) d\mu.
\end{align*}
In the proof we use that $\mu(A^c \cap C_t) = \mu(A \cap C^c_t)$ (this  is because $\mu(A) = \mu(C_t)$).
\end{proof}

\begin{remark}
Applying Lemma \ref{rad-sym-lemma} and  Lemma \ref{ibp-estim} to a radially symmetric measure $\mu = \exp{v(r)} \ dx$
and the function   $F =  v'(r)$ with increasing $v'$
one obtains the following  estimate of the isoperimetric function:
$$
\mathcal{I}_{\mu}(t) \ge \int_{B_{r(t)}} v'(r) \ d\mu,
$$
where $r(t)$ satisfies $\mu(B_{r(t)})=t$.
Note that the term $\frac{d-1}{r}$ is negligible for any increasing $v'$ and large values of $t$. Thus the obtained estimate
is asymptotically sharp for large sets. See also Proposition \ref{bigballs} and Theorem \ref{ball-set-est}
\end{remark}

\begin{example}
\label{exp-alph}
 Let $\mu = e^{r^{\alpha}} dx$, $\alpha>1$.
 For large
values of $\mu(A)$ (say $\mu(A) \ge 1$) one has
$$
\mu^{+}(\partial A) \ge C \mu(A) \log ^{1-\frac{1}{\alpha}} \mu(A).
$$

Indeed, apply Lemma \ref{rad-sym-lemma} to $F = r^{\alpha-1}$.
We get $\mu^{+} (\partial A) \ge \alpha \int_A r^{\alpha-1} \ d\mu \ge \alpha \int_{B_{r(\mu(A))}} r^{\alpha-1} \ d\mu $.
For large $r$ one has
$$
\mu(B_r) \sim c_1 r^{d} e^{r^{\alpha}}, \ \int r^{\alpha-1} \ d \mu \sim  c_2 r^{d+\alpha-1} e^{r^{\alpha}}
$$
and we easily get the desired asymptotics.

Analogously, for  $0 < \tau \le d$
$$
\mu_{\tau} = \frac{dx}{(1-r^2)^{1+ \tau}}, ~ x \in B_1,
$$
satisfies
$$
 \mu^{+}_{\tau}(\partial A)
\ge
C \mu_{\tau}(A)^{1+\frac{1}{\tau}}.
$$
 For $\tau =0$  one has
 $$
 \mu^{+}_0(\partial A) \ge e^{C \mu_0(A)}.
$$
\end{example}

Moreover, applying some refinements of the arguments from above, we show that {\bf for the strictly
log-convex radially symmetric measures the large balls are isoperimetric sets}.

\begin{proposition}
\label{bigballs}
Let $\mu = e^{v(r)} \ dx$
be a radially symmetric measure on $\R^d$ with increasing  $v$. Assume that there exists a smooth function $f : [0,\infty) \to \R$,
satisfying the following assumptions:
\begin{itemize}
\item[1)]
$
|f| \le 1
$
\item[2)]
$
f(r_0)=1
$
\item[3)]
function
$
F(r) = f'(r) + f(r) \Bigl( v'(r) + \frac{d-1}{r} \Bigr)
$
is increasing on $[0,\infty)$.
\end{itemize}
Then among the sets satisfying $\mu(A) = \mu(B_{r_0})$ the ball $B_{r_0}$
has the minimal surface measure $\mu^{+}$.
\end{proposition}
\begin{proof}
Set: $\omega(x) = f(r) \cdot \frac{x}{r} $. Take a set $A$ with $\mu(A)=\mu(B_{r_0})$.
Without loss of generality we assume that  $\partial A$ is smooth and
denote by $n_A$ the outer unit normal to $\partial A$. Applying integration-by-parts we get
\begin{align*}
\mu^{+}(\partial A)
&
\ge
\int_{\partial A} \langle \omega, n_A \rangle e^{v(r)} \ dx
=
\int_{A} \mbox{div} \bigl( \omega \cdot   e^{v(r)}\bigr) \ dx
=
\\&
= \int_{A}  \Bigl[ f'(r) + f(r) \Bigl( v'(r) + \frac{d-1}{r} \Bigr)  \Bigr]  e^{v(r)}\ dx
= \int_{A}  F(r)  e^{v(r)}\ dx.
\end{align*}
Using that $F$ is increasing in $r$, we get by Lemma \ref{rad-sym-lemma} that
$ \int_{A}  F(r)  e^{v(r)} \ dx \ge \int_{B_{r_0}}  F(r)  e^{v(r)}\ dx$.
Thus
$$
\mu^{+}(\partial A)
\ge
\int_{B_{r_0}}  F(r)  e^{v(r)}\ dx.
$$
It remains to note that for $A=B_{r_0}$ we have equalities in all the computations above. Hence
$$
\mu^{+}(\partial A)
\ge
\mu^{+}(\partial B_{r_0}).
$$
\end{proof}

\begin{corollary}
Assume that $v'' \ge 1$.
Let $r_0 \ge \sqrt{d+2}$ .
Then among all the sets satisfying $\mu(A)=\mu(B_{r_0})$
  the ball $B_{r_0}$ has the minimal surface measure.
\end{corollary}
\begin{proof}
Without loss of generality we assume that $v(0)=0$.
Set:
 $$f(r) =\left\{
\begin{array}{rcl}
\frac{3}{2\sqrt{d+2}} \bigl( r - \frac{r^3}{3(d+2)}\bigr) ,  \ \ \ \ \  r \le \sqrt{d+2} \\
  1  \ \ \ \ \ r > \sqrt{d+2}.
\end{array}
\right.
 $$

Assume first that $v=r^2/2$.
 Note that $f$ is continuously differentiable and increasing. One has
 $$
 F(r) =
 \left\{
\begin{array}{rcl}
\frac{3}{2\sqrt{d+2}} \bigl(
d + \frac{2r^2}{3} - \frac{r^4}{3(d+2)} \bigr),  \ \ \ \ \  r \le \sqrt{d+2} \\
  r + \frac{d-1}{r}  \ \ \ \ \ r > \sqrt{d+2},
\end{array}
\right.
 $$

 $$
  F'(r) =
 \left\{
\begin{array}{rcl}
\frac{2r}{\sqrt{d+2}} \bigl(
1-\frac{r^2}{d+2} \bigr),  \ \ \ \ \  r \le \sqrt{d+2} \\
  1 - \frac{d-1}{r^2}  \ \ \ \ \ r > \sqrt{d+2}.
\end{array}
\right.
 $$
  Clearly, $f$ and $F$ satisfy assumptions of the Proposition \ref{bigballs} for every $r_0 \ge \sqrt{d+2}$.

  It is easy to check that for $v$ satisfying $v''>1$ (hence $v' \ge r$)
  one has
  $
    F'(r) \ge
\frac{2r}{\sqrt{d+2}} \bigl(
1-\frac{r^2}{d+2} \bigr)
  $
  for $r \le \sqrt{d+2}$
  and
  $  F'(r) \ge  1 - \frac{d-1}{r^2}$
  for $r \ge \sqrt{d+2}$.
  The proof is complete.
 \end{proof}

\begin{remark}
One can construct more examples using Proposition \ref{bigballs}.
It is applicable under assumption of certain  strict convexity of $v$.
It was pointed out to the author by Frank Morgan that  the arguments of
Proposition \ref{bigballs} imply the following results:
\begin{itemize}
\item[1)] If $v' = r^a$ with $a>0$, then the balls $B_r$ are minimizers if
$r>r_0$, where  $r_0$ satisfies
$r_0 = \sqrt[a+1]{\frac{d+2}{a}}$.
\item[2)]
In $\R^{d}$ with the Riemannian metric
$$
dr^2 + (e^{u(r)})^2 d\theta^2
$$
and density $e^{v(r)}$ (with respect to the Riemannian volume) satisfying
$$
(d-1)u'' + v'' + \frac{d-1}{r^2} \ge 1
$$
the balls $B_r$ with $r \ge \sqrt{d+2}$ are minimizers.
\end{itemize}
\end{remark}

\subsection{An estimate of the isoperimetric function for radially symmetric log-convex measures}

The main aim of this subsection is the following theorem.

\begin{theorem}
\label{ball-set-est}
Let $\mu=e^{v(r)} dx$ be a measure with a convex increasing potential $v$. Then
for every Borel set $A$ the
following inequality holds
$$
\mu^{+}(\partial A) \ge \frac{1}{\sqrt{1 + \pi^2}} \ \mu^{+}(\partial B_r),
$$
provided $\mu(A)=\mu(B_r)$.
\end{theorem}

Roughly speaking, this means that the balls about the origin define the isoperimetric
profile up to some universal constant. In the proof we apply the
mass transportation techniques.

The idea of applying the mass transportation to  isoperimetric
inequalities belongs to M. Gromov. In particular, he applied
the triangular mappings (Knothe mappings) to obtain the classical
isoperimetric inequality. Unfortunately, it seems hard to obtain sharp constants
for isoperimetric inequalities by using only the mass transportation arguments in more general situations. Nevertheless, they can
be used for proving the isoperimetric inequalities with the best
rate. Gromov arguments for the Euclidean isoperimetric problem are
nowadays well-known and can be found in many papers and books. Let us give
 another interesting example.

\begin{example}
\label{hd-iso}
The following isoperimetric inequality holds
in the hyperbolic space $H_d$:
$$
\nu^{+}(\partial A) \ge \max
\Bigl[ \frac{1}{\kappa_d} \nu^{1-\frac{1}{d}}(A), (d-1) \nu(A) \Bigr].
$$
Consider the hyperbolic space $H_{d} = \R^{d-1} \times R^{+}$,
$$
g = \frac{dy_1^2 + \cdots dy_d^2}{y_d^2} =  g_0 (dy_1^2 + \cdots dy_d^2).
$$
Then $
\nu = \frac{d y_1 \cdots dy_d}{y^{d}_d} I_{\{y_d >0\}}.
$
Consider a bounded Borel set $A \subset \{ y_d> \varepsilon\}$  with $\varepsilon >0$.
Let $T$ be the  optimal Euclidean transportation  map pushing forward $\nu|_A$ to the Lebesgue measure
restricted to some Euclidean ball $B_r \subset \R^d$ with a center to be chosen later.
By the change of variables
formula $g^{d/2}_0 = \det D^2 W$
on $A$, hence $
0 = \log \det \Bigl( \frac{1}{\sqrt{g_0}} D^2 W \Bigr)
\le \frac{\mbox{Tr} D^2 W}{\sqrt{g_0}} - d$. Integrating by parts we obtain
$$
d \nu(A) \le \int_{A} \frac{\mbox{Tr} D^2 W}{\sqrt{g_0}} g^{d/2}_{0}\,d\lambda
=
 \int_{\partial A} \frac{\langle n_A, \nabla W \rangle}{g^{1/2}_0} g^{d/2}_{0}\,d\mathcal{H}^{d-1}
 -\int_{A} \langle \nabla W, \nabla g^{\frac{d-1}{2}}_0 \rangle \,d\lambda.
$$
Noting that $
|\nabla_M f|^2_{M} = \frac{1}{g_0} \Bigl( \frac{\partial f}{\partial x_i} \Bigr)^2
$
and estimating $\frac{\langle n_A, T\rangle}{g^{1/2}_0}$
by
$
|T| |n_A|_{M}
$
we get
$
d\nu(A)
\le
\sup |T| \nu^{+}(\partial A)
+
(d-1) \int_{A}\frac{ \langle T, e_d \rangle}{y^d_d} \,d\lambda.
$
Applying the change of variables one gets
$$
d\nu(A)
\le
\sup |T| \nu^{+}(\partial A)
+
(d-1) \int_{B_r} y_d \,d\lambda.
$$
Choosing the center of $B_r$ at the point $(0,-tr)$ with $t \ge 0$ we
get $\sup |T| \le (t+1)r$. In addition, using $\int_{B_r} (y+rt) \,d \lambda =0$, one obtains
$$
d \nu(A) \le (t+1)r \nu^{+}(\partial A)  -  (d-1) \frac{\pi^{d/2}}{\Gamma(1+ \frac{d}{2})} r^{d+1} t.
$$
Taking into account that $\nu(A) = \frac{\pi^{d/2}}{\Gamma(1+ \frac{d}{2})} r^d$
and choosing $t=0, t = +\infty$, one obtains the claim.
\end{example}

\begin{definition}
Let $\mu$, $\nu$ be a couple of probability measures. We say that $T$ is a radial mass transportation of $\mu$ to $\nu$ if
$\nu \circ T^{-1} = \mu$ and it has the form
$$
T(x) = g(r) \cdot N(x)
$$
with $r=|x|$ and  $|N(x)|=1$. In particular,
$T(\partial B_r) \subset \partial B_{g(r)}$.
\end{definition}

There are different ways to transport $\mu$ to $\nu$ by a radial
transportation mapping. Consider the decomposition $\R^d =
[0,\infty) \times S^{d-1}$.
 We use below the following construction.
Let $\nu_r$, $\mu_r$ be the one-dimensional images of $\nu$, $\mu$
under the mapping $x \to |x|= r$.
  Let $g(r)$ be the increasing function pushing forward $\nu_r$ to $\mu_r$. For every fixed $r$
  we
  denote by
 $\nu^{r}(\theta)$,  $\mu^{r}(\theta)$ the corresponding conditional measures
 on $S^{d-1}$ obtained by disintegration of $\nu, \mu$.
 Let
 $$T^{S^{d-1}}_r=T^{S^{d-1}}_r(\theta) :  S^{d-1} \to S^{d-1}$$
  be the optimal transportation mapping pushing forward
 $\nu^{r}(\theta)$ to  $\mu^{r}(\theta)$
 and minimizing the squared Riemannian distance on $ S^{d-1}$.

 Then
 $$
 T = g(r) \cdot T^{S^{d-1}}_r (x/r)
 $$
 is the desired mapping.
 Recall that
 $T^{S^{d-1}}_r$ has the form
 $$
 T^{S^{d-1}}_r = \exp(\nabla_{S^{d-1}} \varphi)
 $$
 for some $\frac{1}{2} d^2$-convex potential $\varphi$. Here $\nabla_{S^{d-1}}$ is the spherical gradient on $S^{d-1}$, $\exp$ is the exponential mapping on $S^{d-1}$ and, in addition,
 \begin{equation}
 \label{tr-manif}
 |\nabla_{S^{d-1}} \varphi(x) | = d(x, T^{S^{d-1}}_r(x)).
 \end{equation}

For a fixed $x$ consider a unit vector $v$ such that $v \bot x$.
One has
$$
\partial_r T = g_r \cdot T^{S^{d-1}}_r + g \cdot \frac{\partial T^{S^{d-1}}_r }{\partial r},
$$
$$
\partial_v T =  g \cdot \frac{\partial T^{S^{d-1}}_r }{\partial v}.
$$

We use below a computation obtained in \cite{CordEras} (pp. 48, 96) (see also \cite{CordEras2}). Consider a mapping $\tilde{T} = \exp(\nabla_{S^{d-1}} \varphi)$.
  Choose for a fixed $x \in S^{d-1}$ an orthonormal basis in the tangent space to $S^{d-1}$
  such that the first vector is equal to $\frac{\nabla_{S^{d-1}} \varphi}{\theta}$, where
   $\theta = |\nabla_{S^{d-1}} \varphi|$. In this basis the Jacobian matrix looks like
  $$
  D\tilde{T} =
\left( \begin{array}{cc}
1+ a &  b^t \\
\frac{\sin \theta}{\theta} b & \cos \theta  \cdot I + \frac{\sin \theta}{\theta} \cdot D
\end{array}
\right),
  $$
  where
  $$
  \mbox{Hess} \ \varphi =
\left( \begin{array}{cc}
a &  b^t \\
b &  D
\end{array}
\right).
  $$
The $d^2$-convexity condition takes the form
$H \ge 0$, where
  $$
  H =
\left( \begin{array}{cc}
1+ a &  b^t \\
b &  \frac{\theta}{\tan \theta} I + D
\end{array}
\right).
  $$
Taking into account that
$B_{r} = r S^{d-1}$
and $
\partial_r T_r ,
\partial_v T_r $ are orthogonal to $T_r$
 one gets
\begin{equation}
\label{chvarshp}
\det DT
=
 g_r \Bigl( \frac{g}{r} \Bigr)^{d-1}
\det DT_r
=g_r \Bigl( \frac{g}{r} \Bigr)^{d-1}
\det
\left( \begin{array}{cc}
1+ a &  b^t \\
\frac{\sin \theta}{\theta} b & \cos \theta  \cdot I + \frac{\sin \theta}{\theta} \cdot D
\end{array}
\right)
\end{equation}

\begin{proof} ({\it Theorem \ref{ball-set-est}}).
 Now fix a Borel set $A$ and take $T$ pushing forward $\mu|_{A}$ to $\mu|_{B_{r_0}}$. Srt:
$$
\mu = \rho(r) \ dx
$$
and apply (\ref{chvarshp})
 \begin{equation}
 \label{cv}
 \rho(g(r)) g_r \Bigl( \frac{g}{r} \Bigr)^{d-1}  \det DT_r  = \rho(r).
 \end{equation}
 Note that the change of variables formula requires some additional justification.
According to the results of Section 3 we can first symmetrize $A$ (see Proposition \ref{symmetry} or \cite{MHH})
and
deal from the very beginning only with sets with a revolution axis.
In this case (\ref{chvarshp}) clearly holds. Indeed,  every $T_r$ is smooth because it is an optimal mapping sending a
spherical cap onto $S^{d-1}$.

 Take the logarithm of the  both sides of (\ref{cv}).
 We apply inequality $\ln \det M \le {\mbox Tr}M-  n$ which holds for every symmetric positive $n \times n$-matrix $M$.
 It is easy to check, that the following inequality holds :
 $$
 \ln \rho - \ln \rho(g) \le \frac{g}{r} \Bigl(1+ a + (d-2) \cos \theta  + \frac{\sin \theta}{\theta} \mbox{Tr} D \Bigr)  + g_r - d.
 $$
 Using the estimates $\frac{\sin \theta}{\theta} \le 1$, $\frac{\theta}{\tan \theta} \le 1$ and positivity of $H$  one obtains
 \begin{align*}
 &
 1+ a +
 (d-2) \cos \theta  + \frac{\sin \theta}{\theta} \mbox{Tr} D =
 \\&
 =
 1+a + \frac{\sin \theta}{\theta} \Bigl( \mbox{Tr} D + \frac{\theta}{\tan \theta} (d-2)  \Bigr)
 \le
 1+a + \Bigl( \mbox{Tr} D + \frac{\theta}{\tan \theta} (d-2)  \Bigr)
 \\&
 \le
   a + d-1 + \mbox{Tr} D = \Delta_{S^{d-1}}  \varphi + d-1.
 \end{align*}
 Let us integrate over $A$ with respect to $\mu = \rho \ dx$ :
 $$
\int_{A} \rho \log \rho \ dx - \int_{A} \rho \log \rho(g) \ dx
\le
\int_{A} \Bigl( \frac{g}{r} (\Delta_{S^{d-1}}  \varphi + d-1) + g_r\Bigr) \rho \ dx -d \mu(A).
 $$
 Applying integration by parts we get
 \begin{align*}
 \int_{A} g_r \rho \ dx & =
 \int_{A} \big\langle \nabla g , \frac{x}{r}  \big\rangle \rho \ dx
 \\&
 =
 \int_{\partial A} g \big\langle n_A , \frac{x}{r}  \big\rangle \ \rho \ d \mathcal{H}^{d-1}
 - (d-1) \int_{A}  \frac{g\rho}{r}    \ dx
 - \int_{A}  g  \rho_r \ dx.
 \end{align*}

 Hence
 $$
 \int_{A} \log\frac{\rho}{\rho(g)} \rho \ dx + d \mu(A)
 + \int_{A}  g  \rho_r \ dx
 \le
 \int_{\partial A} g \big \langle n_A , \frac{x}{r}  \big\rangle \ \rho \ d \mathcal{H}^{d-1}
 +
 \int_{A} \frac{g}{r} \rho \cdot \Delta_{S^{d-1}}  \varphi  \ dx.
 $$

By the coarea formula
$$
 \int_{A} \frac{g}{r} \rho \cdot \Delta_{S^{d-1}}  \varphi  \ dx
 =
 \int_{0}^{\infty} \Bigl[ \frac{1}{r} \int_{\partial B_r \cap A}  \Delta_{S^{d-1}} \varphi \ d \mathcal H^{d-1} \Bigr] \ g \rho \ dr.
 $$

Integrating by parts on $\partial B_r = r S^{d-1}$ we get
that for every smooth $\xi$
$$
\int_{\partial B_r}  \Delta_{S^{d-1}} \varphi \ \xi \ d \mathcal H^{d-1}
=
-
\int_{\partial B_r}  \langle \nabla_{S^{d-1}} \varphi, \nabla_{S^{d-1}}  \xi \rangle \ d \mathcal H^{d-1}.
$$
Note that $\frac{1}{r} \nabla_{S^{d-1}} \xi$ is nothing else but the projection $Pr_{TS^{d-1}} \nabla \xi$  of the $\nabla \xi$ onto the tangent space to
$\partial B_r$.
Hence
\begin{align*}
 \int \frac{g}{r} \rho \xi \cdot \Delta_{S^{d-1}}  \varphi  \ dx
 =
 &
 -
\int_{0}^{\infty} \Bigl[  \int_{\partial B_r } \langle Pr_{TS^{d-1}} \nabla \xi, \nabla_{S^{d-1}}  \varphi \rangle \ d \mathcal H^{d-1} \Bigr] \ g \rho \ dr
\\&
  = -
\int \langle Pr_{TS^{d-1}} \nabla \xi, \nabla_{S^{d-1}}  \varphi \rangle  \ g \rho \ dx.
 \end{align*}

Approximating $I_A$ by smooth functions we get
$$
 \int_{A} \frac{g}{r} \rho  \cdot \Delta_{S^{d-1}}  \varphi  \ dx
 =
\int_{\partial A} \langle Pr_{TS^{d-1}} n_A, \nabla_{S^{d-1}}  \varphi  \rangle  g \rho \ d \mathcal{H}^{d-1}.
$$

Hence
 \begin{align*}
 &
 \int_{A} \rho \log\frac{\rho}{\rho(g)}  \ dx + d \mu(A)
 + \int_{A}  g  \rho_r \ dx
 \le
 \\&
 \le
 \int_{\partial A} g \big \langle n_A , \frac{x}{r}  \big\rangle \ \rho \ d \mathcal{H}^{d-1}
 +
\int_{\partial A} \langle Pr_{TS^{d-1}} n_A, \nabla_{S^{d-1}}  \varphi  \rangle  g \rho \ d \mathcal{H}^{d-1}.
 \end{align*}

Since $\exp(\nabla_{S^{d-1}}  \varphi)$ takes values in the unit sphere, one has $|\nabla_{S^{d-1}}  \varphi| \le \pi$ (see (\ref{tr-manif})) and the right-side does not exceed
$$
               \int_{\partial A} g \sqrt{1 + \pi^2} \rho \  d \mathcal{H}^1
                 \le r_0 \sqrt{1 +  \pi^2} \mu^{+}(\partial A).
$$
Note that $g(r) \le r$. Hence
 $$
  \int_{A} \log\frac{\rho}{\rho(g)} \rho \ dx  \ge 0
  $$
  and
  $$
   \int_{A}  g  \rho_r \ dx = \int_{A}  g  v_r \ d\mu
   =  \int_{B_{r_0}}  r  v_r(g^{-1}) \ d \mu
   \ge
   \int_{B_{r_0}}  r  v_r \ d \mu.
  $$
  Finally, we obtain
  $$
  d \mu(B_{r_0})
 +  \int_{B_{r_0}}  r  v_r \ d \mu
   \le r_0 \sqrt{1 +  \pi^2} \mu^{+}(\partial A).
  $$
 The divergence theorem implies  that the left-hand side is equal to
  $r_0 \mu^{+}(\partial B_{r_0})$.
 Hence
  $$
  \mu^{+}(\partial B_{r_0}) \le
  \sqrt{1 +  \pi^2} \mu^{+}(\partial A).
  $$
\end{proof}

\subsection{Product measures. A comparison theorem.}

We start this subsection with a comparison result. The comparison
theorems are very important tools for studying  the isoperimetric
estimates. The most well-known example is the Levy-Gromov's isoperimetric inequality  for the Ricci positive
manifolds. Its probabilistic
version is given by the Bakry-Ledoux comparison theorem  \cite{BakLed} (see also
\cite{MorganMWD}).

{\bf Theorem (Bakry-Ledoux):}
Assume that
$$
\mu = e^{-V} \ dx,
$$
is a probability measure with $V$ satisfying
$$
D^2 V \ge c \cdot \mbox{Id}, \ \ c>0
$$
and
$\gamma_c$ is the Gaussian measure with the covariance operator $c \cdot \mbox{Id}$.
Then
$$
\mathcal{I}_{\mu} \ge \mathcal{I}_{\gamma_c}.
$$

The Bakry-Ledoux theorem is an immediate corollary of the following result:

{\bf Theorem (Caffarelli):} For every probability measure $\mu = e^{-V} \ dx$
with $D^2 V \ge I$ the optimal transportation mapping $T = \nabla \varphi$
with convex $\varphi$ which  pushes forward
the standard Gaussian measure $\gamma$ onto $\mu$  is $1$-Lipschitz
(see \cite{Caff}, Theorem 11 and recent developments in \cite{BarKol},
\cite{Kol}, \cite{Vald}, \cite{KimMilman}).

Note that the spaces from these examples are positively curved
(i.e. with a positive Bakry-Emery tensor). Concerning the
negatively curved spaces,  it is still an open problem,
whether the Cartan-Hadamard conjecture holds  in general case.

{\bf Cartan-Hadamard conjecture:} Let $M$ be a complete, smooth, simply connected Riemannian manifold with
 sectional curvatures bounded from above by
a constant nonpositive value $c$. The isoperimetric function $\mathcal{I}_M$
satisfies $\mathcal{I}_M \ge \mathcal{I}_c$, where $\mathcal{I}_c$
is the isoperimetric function of the model
space with the constant sectional curvature $c$.

The conjecture is known to be true for a long time for $d=2$  (see, for instance,  \cite{W}).
Other known cases: $d=3$ (B.~Kleiner, \cite{Klein}), $d=4, \ c=0$ (C.~Croke, \cite{Croke}).
Some new proofs and recent developments can by found in
M.~Ritor{\'e} \cite{Ritore}, F.~Schulze \cite{Schulze}.

In this paper we prove a comparison result for the products of
log-convex measures. It turns out that a natural model measure
for the one-dimensional log-convex distributions  has the following form:
$$
\nu_{A} = \frac{dx}{\cos Ax}, \ -\frac{\pi}{2A}  < x < \frac{\pi}{2A}.
$$
Its potential $V$  satisfies
$$
V'' e^{-2V} =A^2.
$$

By a result from \cite{RCBM} (Corollary 4.12) the isoperimetric sets for strictly
log-concave even measures on the line are symmetric intervals containing the
origin.

Using this result it is easy to compute
the isoperimetric profile of $\nu_A$:
$$
\mathcal{I}_{\nu_A}(t) =  e^{At/2}+e^{-At/2}.
$$

It turns out that in the log-convex case the following quantity is a natural measure of convexity of the potential:
$$
W^{''} e^{-2W}
$$
(unlike $W^{''}$ in the probabilistic case).

We establish here the following  analog of the Caffarelli result.

\begin{proposition}
Let $\mu = e^{W} dx$ be a measure with even convex potential $W$.
Assume that
$$W''e^{-2W} \ge A^2,$$
and $W(0)=0$. Then $\mu$ is the image of $\nu_A$ under
a $1$-Lipschitz increasing mapping.
\end{proposition}
\begin{proof}
Without loss of generality one can assume that $W$ is smooth and $W''e^{-2W} > A^2$.
Let $\varphi$ be a convex potential such that
$T=\varphi'$ sends $\mu$ to $\nu_A$.
In addition, we require that $T$ is antisymmetric.
Clearly, $\varphi'$ satisfies
$$
e^{W}  = \frac{\varphi''}{\cos A \varphi'}.
$$
Assume that $x_0$ is a local maximum point for $\varphi''$.
Then at this point
$$
\varphi^{(3)}(x_0)=0 \ \ \varphi^{(4)}(x_0) \le 0.
$$
Differentiating the change of variables formula at $x_0$ twice we get
$$
W''
= \frac{\varphi^{(4)}}{\varphi''}
- \Bigl( \frac{\varphi^{(3)}}{\varphi''}\Bigr)^2 + \frac{A^2}{\cos^2 A \varphi'}
(\varphi'')^2 + A \frac{\sin A \varphi'}{\cos A \varphi'} \varphi^{'''}.
$$
Consequently one has at $x_0$
$$
W''
\le \frac{A^2}{\cos^2 A \varphi'} (\varphi'')^2 = A^2 e^{2W}.
$$
But this contradicts to the main assumption.

Hence $\varphi''$ has no local maximum. Note that $\varphi$ is even. This implies that that $0$
is the global minimum of $\varphi''$. Hence
$$
\varphi'' \ge \varphi''(0)=1.
$$
Clearly, $T^{-1}$ is the desired mapping.
\end{proof}

The corollary below can be seen as an elementary "flat" version of the
Cartan-Hadamard-type comparison results or as a log-convex version
of the Bakry-Ledoux comparison theorem (see also \ref{log-convex-euclid}).

\begin{corollary}
Let $\mu$ be a product measure
$$
\mu = \prod_{i=1}^d e^{W_i} \ dx_i
$$
with
$$
W^{''}_i e^{-2W_i} \ge A^2, \ \ \mbox{\rm $W_i$ is even  and} \ W_i(0)=0.
$$
Then
$$
\mathcal{I}_{\mu} \ge \mathcal{I}_{\mu_A},
$$
where $\mu_A = \prod_{i=1}^{d} \frac{dx_i}{\cos Ax_i}$
is the measure on $[-\frac{\pi}{2A}, \frac{\pi}{2A}]^d$.
\end{corollary}

It is not clear, whether this result can be generalized to the
multi-dimensional case. More generally, it is not clear, which
measure of convexity $V$ is responsible for the isoperimetric
properties (in the probabilistic case this is the Hessian of the
potential). Surprisingly, in certain situation some lower bounds
on
$$\det D^2 V \cdot e^{-V}$$ turn out to be sufficient for some isoperimetric estimates.

We denote by $\kappa_d$ the constant appearing in the Euclidean isoperimetric inequality
$\lambda(A)^{1 - \frac{1}{d}} \le \kappa_{d} \mathcal{H}^{d-1}(\partial A)$.

\begin{proposition}
Let $V \ge 0$ be convex and, in addition,
\begin{equation}
\label{prod-det}
e^{-V} \det D^2 V \ge K^{d}
\end{equation}
for some $K \ge 0$.
Then for some constant $C(d)$ the following inequality holds
\begin{equation}
\label{gbi-rd}
\mu(A)^{1-\frac{1}{d}}
+   K  C(d) \ \mu(A)^{1+\frac{1}{d}}
\le
 \kappa_d \mu^{+}(\partial A).
 \end{equation}
\end{proposition}
\begin{proof}
Let $\nabla W$ be the  optimal transportation pushing forward
$\mu|_{A}$ to $\lambda|_{B_r}$.
By the change of variables formula (see \cite{McCann97})
$$
e^{V} = \det D^2W_a
$$
on $I_{A}$, where $D^2 W_{a}$ is the second Alexandrov derivative
of $W$ (recall that $D^2 W \ge D^2_a W$, where $D^2 W$ is the distributional derivative). Taking the
logarithm of the both sides and applying the standard estimate one
gets
$$
0 \le V \le \Delta W_a - d.
$$
Integrating over $A$ one gets
$$
d \mu(A) \le  \int_{A} \Delta W \,d\mu
=
 \int_{\partial A} \langle n_A, \nabla W \rangle e^{V} \ d \mathcal{H}^{d-1} -
 \int_{A} \langle \nabla W,\nabla V \rangle e^V \,d\lambda.
$$
Hence
\begin{equation}
\label{key_est}
d\mu(A) +  \int_{A} \langle \nabla W,\nabla V \rangle e^V \,d\lambda
\le
 r \mu^{+}(\partial A).
\end{equation}

By the change of variables
  $\int_{A} \langle \nabla W,\nabla V \rangle e^V \,d\lambda=
\int_{B_r} \langle x,\nabla V \circ \nabla \Phi \rangle  \,d\lambda,
$
where $\Phi = W^{*}$ is the corresponding convex conjugated function.
Note that $\nabla \Phi$ is the optimal transport of $\lambda|_{B_r}$  onto $\mu|_{A}$.
Taking into account that $x = \nabla  \frac{|x|^2 - r^2}{2}$ and integrating by parts
one gets
\begin{align*}
\int_{B_r} \langle x,\nabla V \circ \nabla \Phi \rangle  \,d\lambda &
\ge
\frac{1}{2} \int_{B_r} \bigl( r^2 - |x|^2\bigr) \mbox{Tr} \Bigl[ D^2 V \bigl( \nabla \Phi\bigr) \cdot D^2_a \Phi  \Bigr] \,d\lambda
\\&
=
\frac{1}{2} \int_{B_r} \bigl( r^2 - |x|^2\bigr) \mbox{Tr} \Bigl[ (D^2_a \Phi)^{\frac{1}{2}} \cdot D^2 V \bigl( \nabla \Phi\bigr) \cdot (D^2_a \Phi)^{\frac{1}{2}}  \Bigr] \,d\lambda.
\end{align*}
Since $\nabla \Phi$ pushes forward $\lambda|_{B_r}$ to $\mu|_{A}$, by the change of variables
$$
e^{V(\nabla \Phi)} \det D^2_a \Phi =1.
$$
Note that $(D^2_a \Phi)^{\frac{1}{2}} \cdot D^2 V \bigl( \nabla \Phi\bigr) \cdot (D^2_a \Phi)^{\frac{1}{2}} $ is a symmetric matrix.
It is nonnegative, since $V$ and $\Phi$ are convex.
Hence
\begin{align*}
&
\frac{1}{d} \mbox{Tr} \Bigl[ (D^2_a \Phi)^{\frac{1}{2}} \cdot D^2 V \bigl( \nabla \Phi\bigr) \cdot (D^2_a \Phi)^{\frac{1}{2}}  \Bigr]
\ge  \Bigl( \det D^2_a \Phi \cdot \det D^2 V \bigl( \nabla \Phi\bigr) \Bigr)^{\frac{1}{d}}
=
\\&
= \Bigl( e^{-V(\nabla \Phi)} \cdot \det D^2 V \bigl( \nabla \Phi\bigr) \Bigr)^{\frac{1}{d}}
\ge
K.
\end{align*}
Finally we obtain that for some constant $C$ depending only on $d$
$$
d\mu(A) +  \frac{d K C}{2}  r^{d+2}
\le
 r \mu^{+}(\partial A).
$$
The desired result follows from the relation
 $r = \Bigl(\frac{\mu(A)}{\lambda(B_1)}\Bigr)^{\frac{1}{d}} = d \kappa_d \mu^{\frac{1}{d}}(A)$.
\end{proof}

\begin{corollary}
\label{log-convex-euclid}
If $V \ge 0$ and convex, then $\mu= e^V \ dx$
satisfies the Euclidean isoperimetric inequality.
\end{corollary}

\begin{example}
Note that condition (\ref{prod-det}) is tensorizable.
The following product measure on $[-\pi/2,\pi/2]^d$ satisfies inequality (\ref{gbi-rd}) :
$$
\mu = \prod_{i=1}^{d} \frac{d x_i}{\cos^2 x_i}.
$$
\end{example}

We finish this section with another isoperimetric estimate  for log-convex product measures.

\begin{theorem}
Consider a log-convex measure
$$
\mu = \prod_{i=1}^{d} e^{V_i(x_i)} \ dx_i
$$
such that every  $V_i$ is convex, even and $V_i(0)=0$. Assume that
there exists a concave increasing function $G: [0, \infty) \to [0,
\infty)$ satisfying
$$
e^{-V_i} G'(V'_i) V^{''}_i \ge 1.
$$
Then for some constants $c_1(d), c_2(d)$ the following inequality
holds $$ \mu(A)^{1-\frac{1}{d}} +   c_1(d) \ \mu(A) \cdot
G^{-1}\bigl(c_2 \mu^{1/d}(A)\bigr) \le
 \kappa_d \mu^{+}(\partial A).
 $$
\end{theorem}
\begin{proof}
According to the general result on Steiner symmetrization for
product measures (see \cite{Ros}), the Steiner symmetrization with respect to any axis does
not increase the surface measure of the  set. Since the family of
symmetric intervals are isoperimetric sets for every
one-dimensional measure $e^{V_i} \ dx_i$, we can assume from the
very beginning that $A$ is symmetric with respect to every mapping
$x \to (\pm x_1, \cdots, \pm x_i, \cdots, \pm x_n)$. Let $\nabla W$ be the  optimal
transportation pushing forward $\mu|_{A}$ to $\lambda|_{B_r}$. In
the same way as in the previous proposition we prove
$$
d\mu(A) +  \int_{A} \langle \nabla W,\nabla V \rangle e^V \,d\lambda
\le
 r \mu^{+}(\partial A)
$$
and
$$
\int_{A} \langle \nabla W,\nabla V \rangle e^V \,d\lambda=
\int_{B_r} \langle x,\nabla V \circ \nabla \Phi \rangle
\,d\lambda.
$$
By the symmetry reasons the functions
$x_i$ and $V'_i(\nabla \Phi) = V'_i(\Phi_{x_i}) $
have the same sign.
Hence by the Jensen inequality (which is applicable because $G^{-1}$ is convex)
\begin{align*}
\int_{B_r} \langle x,\nabla V \circ \nabla \Phi \rangle  \,d\lambda &
= \sum_{i=1}^{d}
\int_{B_r} |x_i|  |V'_i(\Phi_{x_i})| \,d\lambda
\ge
\\& \ge
\sum_{i=1}^{d}
\int_{B_r} |x_i| \ d\lambda \cdot
G^{-1} \Bigl( \frac{\int_{B_r} |x_i|  G(V'_i(|\Phi_{x_i}|)) \ d\lambda }{\int_{B_r} |x_i| \ d \lambda} \Bigr).
\end{align*}
Denote  $S = (G(V'_i)).$
One has
$$
\sum_i
\int_{B_r} |x_i|  G(V'_i(|\Phi_{x_i}|)) \,d\lambda =
{\int_{B_r}} \langle x,   S \circ \nabla \Phi \rangle \ d \lambda .
$$
The latter is larger than
\begin{align*}
\int_{B_r} \langle x,\nabla S \circ \nabla \Phi \rangle  \,d\lambda &
\ge
\frac{1}{2} \int_{B_r} \bigl( r^2 - |x|^2\bigr) \mbox{Tr} \Bigl[ D^2 S \bigl( \nabla \Phi\bigr) \cdot D^2_a \Phi  \Bigr] \,d\lambda
\\&
=
\frac{1}{2} \int_{B_r} \bigl( r^2 - |x|^2\bigr) \mbox{Tr} \Bigl[ (D^2_a \Phi)^{\frac{1}{2}} \cdot D^2 S \bigl( \nabla \Phi\bigr) \cdot (D^2_a \Phi)^{\frac{1}{2}}  \Bigr] \,d\lambda.
\end{align*}
Note that $DS$ is diagonal and nonnegative.
Applying the arithmetic-geometric inequality for the trace and determinant
we get
$$
\int_{B_r} \langle x,\nabla S \circ \nabla \Phi \rangle \,d\lambda
\ge \frac{d}{2} \int_{B_r} \bigl( r^2 - |x|^2\bigr) \bigl[
\mbox{det} (D^2_a \Phi) \cdot \det D^2 S \bigl( \nabla \Phi\bigr)
\bigr]^{1/d} \,d\lambda.
$$
The latter is equal to
$$
\frac{d}{2} \int_{B_r} \bigl( r^2 - |x|^2\bigr)  \Bigl[ e^{-V(\nabla \Phi)}\det D^2 S \bigl( \nabla \Phi\bigr) \Bigr]^{1/d}  \,d\lambda.
$$
Note that
$$
e^{-V(\nabla \Phi)}\det D^2 S \bigl( \nabla \Phi\bigr)
=
\Bigl[  \prod_{i=1}^{d} e^{-V_i} G'(V'_i) V^{''}_i\Bigr] \circ \nabla \Phi
\ge 1.
$$
Hence for some $C(d)$
$$
\int_{B_r} \langle x,\nabla S \circ \nabla \Phi \rangle  \,d\lambda
\ge C(d) r^{d+2}.
$$
Thus
$$
\int_{B_r} \langle x,\nabla V \circ \nabla \Phi \rangle  \,d\lambda
\ge C_1(d) r^{d+1} G^{-1}( C_2(d) r).
$$
The result follows from the relation $\mu(A) = \lambda(B_r)$.
\end{proof}

\begin{corollary}
The measure
$$
\mu = \prod_{i=1}^{d} \frac{d x_i}{\cos x_i}.
$$ on
 $[-\pi/2,\pi/2]^d$
 satisfies
 $$
 \mu(A)^{1-\frac{1}{d}}
 + C_1 e^{C_2 \mu^{1/d}(A)}
 \le  \kappa_d \mu^{+}(\partial A)
 $$
 (one can take $G=\ln(x+\sqrt{1+x^2})$).
\end{corollary}

\end{document}